\documentclass[11pt,reqno]{amsart}
\usepackage{graphicx}
\usepackage{color}
\usepackage{amsmath,amssymb,amsthm,amsfonts}
\usepackage{graphicx}
\usepackage{color}
\usepackage{multicol}
\usepackage{mathptmx}
\def\R{\mathbb{R}}
\makeatletter

\newcommand{\Rmnum}[1]{\expandafter\@slowromancap\romannumeral #1@}
\makeatother
\newcommand{\D}{\displaystyle}
\newtheorem{thm}{Theorem}[section]
\newcommand{\norm}[1]{\left\lVert#1\right\rVert}
\newtheorem{definition}{Definition}[section]

\newtheorem{lemma}[thm]{Lemma}
\newtheorem{remark}{Remark}[section]
\newtheorem{theorem}[thm]{Theorem}
\newtheorem{proposition}[thm]{Proposition}

\newcommand{\abs}[1]{\left\vert#1\right\vert}
\usepackage[numbers,sort&compress]{natbib}
\begin{document}

\author{Hongyun Peng$^*$}
\thanks{$^*$ School of Applied Mathematics, Guangdong University of Technology, Guangzhou, 510006,
China; email: penghy010@163.com}

\author{Zhian Wang$^\dag$}
\thanks{$^\dag$ Department of Applied Mathematics, Hong Kong Polytechnic University, Hung Hom, Kowloon, Hong Kong; email: mawza@polyu.edu.hk}

\title[Parabolic-Hyperbolic System with Discontinuous Data]{On a Parabolic-Hyperbolic Chemotaxis System with Discontinuous Data: Well-posedness, Stability and Regularity}

\begin{abstract}
The global dynamics and regularity of parabolic-hyperbolic systems is an interesting topic in PDEs due to the coupling of competing dissipation and hyperbolic effects. This paper is concerned with the Cauchy problem of a parabolic-hyperbolic system derived from a chemotaxis model describing the dynamics of the initiation of tumor angiogenesis. It is shown that, as time tends to infinity, the Cauchy problem with large-amplitude discontinuous data admit global weak solutions which converge to a constant state (resp. a viscous shock wave) if the asymptotic states of initial values at far field are equal (resp. unequal). Our results improve the previous results where initial value was required to be continuous and have small amplitude. Numerical simulations are performed to verify our analytical results, illustrate the possible regularity of solutions and speculate the minimal regularity of initial data required to obtain the smooth (classical) solutions of the concerned parabolic-hyperbolic system.

\vspace{0.2cm}

\vspace{0.2cm}

\end{abstract}

%
%

\maketitle

\numberwithin{equation}{section}
\bigbreak
\section{Introduction}
The parabolic-hyperbolic coupled system of partial differential equations may arise from physics, mechanics and material science such as the
compressible Navier-Stokes equations, thermo(visco)elastic systems and elastic systems. The properties of solutions to nonlinear parabolic-hyperbolic
coupled systems are very different from those of parabolic or hyperbolic equations. There are many mathematical researches for various parabolic-hyperbolic coupled systems  on the well-posedness (local and global) and asymptotical behavior of solutions since 1970s (cf. \cite{Qin, Zheng}). It is well-known that the diffusion (parabolic) dissipation can smoothen solutions from the crude initial data, while on the contrary the hyperbolic effect can coarsen solutions from the smooth initial data (e.g. see \cite{smoller}).   Therefore the relation between the regularity of solutions and the initial values of parabolic-hyperbolic systems has been an interesting topic and attracted many studies (e.g. see \cite{Qin, HaoJDE, Frid}). This paper is concerned with the Cauchy problem of the following parabolic-hyperbolic system in $\R$:
\begin{equation}\label{hp}
\begin{cases}
u_t-\chi(uv)_x=Du_{xx},\\
v_t-u_x=0,
\end{cases}
\end{equation}
with the initial value
\begin{equation}\label{Initial} (u, v)(x, 0)=(u_0, v_0)(x), \ \ x\in
{\mathbb{R}}
\end{equation}
and the far-field behavior (i.e., asymptotic state at $\pm \infty$):
\begin{equation}\label{Boundary}
 (u, v)(\pm\infty, t)=(u_\pm, v_\pm),
\end{equation}
where $u_{\pm}\geq  0$, $\chi$ and $D$ are positive constants. The system \eqref{hp} is transformed from the following PDE-ODE singular chemotaxis model proposed in \cite{LSN} (see \cite{Levine97, Othmer97} for mathematical derivation) to describe the interaction between signaling molecules vascular endothelial growth factor (VEGF) and vascular endothelial cells during the initiation of tumor angiogenesis,
\begin{eqnarray}\label{ks1}
\left\{\begin{array}{lll}
u_t= (Du_x-  \xi u (\ln c)_x)_x,\\
c_t=- \mu u c,
\end{array}\right.
\end{eqnarray}
via a Cole-Hopf type transformation
 \begin{equation*}\label{ch}
v=-\frac{1}{\mu}(\ln c)_x =-\frac{1}{\mu}\frac{c_x}{c}, \ \ \chi=\mu\xi>0,
 \end{equation*}
where $u(x,t)$ and $c(x,t)$ denote the density of vascular endothelial cells and concentration of VEGF, respectively; $D>0$ is the diffusivity of endothelial cells, $\xi>0$ is referred to as the chemotactic coefficient measuring the intensity of chemotaxis and $\mu$ denotes the degradation rate of the chemical $c$. Due to the challenge of logarithmic singularity in \eqref{ks1}, most of mathematical studies in the literature focus attention on the non-singular transformed system \eqref{hp}.

The Cauchy problem (\ref{hp})-(\ref{Boundary}) has received a lot attention in the literature.
When the left and right asymptotic states are identical $(u_-=u_+, v_-=v_+)$, the global existence and long-time behavior of (strong) solutions of (\ref{hp}) in $\R$ have been obtained in \cite{GXZZ, Lipz, ZhangYH, zhangts}. When the left and right asymptotic states are different $(u_-\neq u_+, v_-\neq v_+)$, the existence of traveling wavefront solutions of (\ref{hp}) was obtained first in \cite{wang08} and nonlinear stability of traveling wave solutions was subsequently established  in a series of works \cite{jin13, Li09, Li10}. The stability of composite waves of (\ref{hp}) in $\R$ was proved in \cite{Lij13}. All these works have assumed initial values have $H^1$ or higher regularity and show that the strong solutions may have the same regularity as initial values, where in particular smooth solutions can be obtained if the initial values have $H^2(\R)$-regularity (cf. \cite{GXZZ, ZhangYH, zhangts}).
Then an interesting question is whether global strong solutions of (\ref{hp}) can  be obtained if the initial value has lower regularity such as $L^p$-regularity ($1\leq p\leq \infty$). The answer seems unclear due to the coupling of parabolic and  hyperbolic equations.

In this paper, we shall prove if the initial value $(u_0, v_0)$ has only $L^p$-regularity, global existence and stability of weak solutions can be established.  We show our results in two cases with initial data in $L^p$-space which essentially include discontinuous functions. First if the left and right asymptotic states are identical, we prove that the Cauchy problem (\ref{hp})-(\ref{Boundary}) admits global weak solutions which converge to the asymptotic states in some sense as time tend to infinity (see Theorem \ref{pw6-th}). Second if the left and right asymptotic states are different, we show that (\ref{hp})-(\ref{Boundary}) also admits global weak solutions which asymptotically converge to a (shifted) traveling wave solution in appropriate functional space (Theorem \ref{u_+=0}). In both cases, the initial values are allowed to have large amplitudes and we show that the solution component $u$ is spatially continuous for any $t>0$ which has higher regularity than initial values but the solution component $v$ only has the same regularity as initial values.  However we are unable to prove whether the discontinuity of solution component $v$ persists in time if initial values are discontinuous or $(u,v)$ can have higher regularity. To speculate possible outcomes, we use numerical simulations to illustrate that if the initial value $(u_0,v_0)\in L^p(\R)$ is discontinuous, classical solutions is impossible but $H^1$(or continuous) solutions appear to be attainable for the solution component $v$. Furthermore we numerically find that the solution $(u,v)$ of \eqref{hp} will be smooth as long as the initial value $(u_0,v_0)$ has $H^1$-regularity. These numerical evidences indicate that the coupled parabolic-hyperbolic system \eqref{hp} can not smoothen solutions from the discontinuous initial data due to the hyperbolic effect,  but can slightly improve the regularity due to the parabolic dissipation. We also see from numerical simulations that the minimal regularity of initial values to obtain classical (smooth) solutions for \eqref{hp} seems to be $H^1(\R)$.  However all these speculations lack of justification and leave us interesting questions to pursue in the future. Since the initial value $(u_0,v_0)$ considered in the current work has only $L^p(\R)$-regularity,  the energy estimate framework in previous works (cf. \cite{Lipz, zhangts,jin13, Li09, Li10}) relying on the higher regularity of initial values no longer applies. We have employed a few new approaches, such as mollifying technique, time-weight function and effective viscous flux, to obtain desired results (see details in Remark \ref{remtech}). We should remark that the global dynamics of PDEs with discontinuous data is an important topic arising from fluid mechanics and gas dynamics to understand that how the discontinuities evolve in the fluid. Hoff has contributed a series of important results to this direction  (cf. \cite{Hoff-1987, Hoff-19980, Hoff-1992, Hoff-jde, Hoff-arma}) with further development in \cite{Hoff-1989, zhang131, zhang132},  which have essentially inspired our current work.

Before concluding the introduction, we briefly recall some other results related to the system (\ref{hp}).
First in the one dimensional bounded interval, when the Neumann-Dirichlet mixed boundary conditions are imposed, the global existence of solutions of (\ref{hp}) was first established in \cite{zhang07} for small initial data and later in \cite{Li-Zhao-JDE, Li112} for large initial data, where the Dirichlet problem was also considered in \cite{Li-Zhao-JDE}. For the multidimensional whole space $\R^d$ ($d\geq 2$), when the initial datum is close to a constant ground state $(\bar{u}, {\bf 0})$, numerous results have been  obtained.  First a blowup criteria of solutions was established in \cite{Fan-zhao, Li111} and  long-time behavior of solutions was obtained in \cite{Li111} if $(u_0-\bar{u} , {\bf v_0})\in H^s(\R^d)$ for $s>\frac d2+1$ and $\norm{(u_0-\bar{u} , {\bf v_0})}_{H^s\times H^s}$ is small. Later, Hao \cite{Hao} established the global existence of mild solutions in the critical Besov space $\dot{B}_{2,1}^{-\frac12}\times (\dot{B}_{2,1}^{\frac12})^d $ with minimal regularity in the Chemin-Lerner space framework. The global well-posedness of strong solutions of \eqref{hp} in $\R^3$ was established in \cite{DL} if $\norm{(u_0-\bar{u} , {\bf v_0})}_{L^2\times H^1}$ is small. If the initial datum has a higher regularity such that $\norm{(u_0-\bar{u} , {\bf v_0})}_{H^2\times H^1}$ is small, the algebraic decay of solutions was further derived in \cite{DL}. Recently, Wang, Xiang and Yu \cite{Wang-xiang-yu} established the global existence and time decay rates of solutions of \eqref{hp} in $\R^d$ for $d=2,3$ if $(u_0-\bar{u} , {\bf v_0})\in H^2(\R^d)$ and $\norm{(u_0-\bar{u} , {\bf v_0})}_{H^1\times H^1}$ is small.  In the multidimensional bounded domain $\Omega \subset \R^d(d=2,3)$,  global existence and exponential decay rates of solutions under Neumann boundary conditions were obtained in \cite{Li112} for small data, and local existence of solutions in two dimensions with Dirichlet boundary conditions was given in \cite{HWJMPA}. Finally we mention that when the Laplacian (diffusion) in \eqref{hp} was modified to a fractional Laplacian, the global existence of solutions of  \eqref{hp} in a torus with periodic boundary conditions in some dissipation regimes was established in \cite{Rafael1,Rafael2}.

The rest of paper is organized as follows. In section 2,  we state our main results. In section 3, we collect some
elementary facts and inequalities which will be needed in later analysis. In section 4, we prove the large-time behavior of solutions with constant states. The proof of nonlinear stability of viscous shock waves is given in section 5. In section 6, we perform numerical simulations to verify our analytical results and speculate the possible regularity of solutions.

\section{Statement of  main results}
We first explain some  conventions used throughout the paper.  $C$
denotes a generic positive constant which can change from one line
to another. $H^k(\mathbb{{R}})$ denotes the usual $k$-th order
Sobolev space on $\mathbb{R}$ with norm $
\|f\|_{H^k(\mathbb{R})}:=\left(\sum_{j=0}^{k}\int_{\R}
|\partial_x^jf|^2dx\right)^{1/2}$.
For simplicity, we denote $\|\cdot\|:=\|\cdot\|_{L^2(\mathbb{R})}$
and
$\|\cdot\|_k:=\|\cdot\|_{H^k(\mathbb{R})}$.

Next, we shall present our main results concerning the asymptotic behavior of solutions of the Cauchy problem \eqref{hp}-\eqref{Boundary}.

\subsection{Constant states} We first consider the case where the end states $(u_-,v_-)$ and $(u_+, v_+)$ are connected by a constant, say $(u_-,v_-)=(u_+, v_+)=(\bar{u}, \bar{v})=(1,0)$. To state our results on the global stability of the constant steady state $(1,0)$, we first present the definition of weak solutions of \eqref{hp}-\eqref{Boundary}.

\begin{definition} We say that $(u, v)$ is a weak solution of \eqref{hp}-\eqref{Boundary}, if $(u,v)$ is suitably
integrable, and for all test functions $\Psi\in C_0^\infty(\mathbb{R}\times [0, \infty))$ satisfy that
\begin{equation*}
\begin{split}
\int_{\mathbb{R}}u_0\Psi_0(x)dx+\int_0^\infty\int_{\mathbb{R}}\left(u\Psi_t-Du_x\Psi_x\right)dxdt=\chi \int_0^\infty\int_{\mathbb{R}}uv\Psi_x dxdt
\end{split}
\end{equation*}
and
\begin{equation*}
\begin{split}
\int_{\mathbb{R}}v_0\Psi_0(x)dx+\int_0^\infty\int_{\mathbb{R}}\left(v\Psi_t-u_x\Psi\right)dxdt=0.
\end{split}
\end{equation*}
\end{definition}

Then our first main result is encompassed in the following theorem.
\begin{theorem}\label{pw6-th}
Suppose that the initial
data satisfy
\begin{equation}\label{pw6-ict1}
u_0-1 \in L^2(\mathbb{R})\cap L^4(\mathbb{R}), \ \ v_0 \in L^2(\mathbb{R})\cap L^\infty(\mathbb{R}), \ \  u_0>0.
\end{equation} Then the Cauchy
problem \eqref{hp}-\eqref{Boundary} has a
 global weak solution $(u,v)(x,t)$ satisfying
\begin{equation}\label{pw6-le-last0}
\begin{cases}
\begin{split}
&u-1 \in  L^\infty([0,\infty); L^2(\mathbb{R}))\cap C((0,\infty);C(\mathbb{R})), \ \  u_x\in L^2([0,\infty);L^2(\mathbb{R})),\\
&v\in L^\infty([0,\infty); L^2(\mathbb{R})\cap L^{\infty}(\mathbb{R}))\cap L^6([0,\infty); L^6(\mathbb{R})).
\end{split}
\end{cases}
\end{equation}
Furthermore, the following convergence holds:
 \begin{equation}\label{pw6-lb}
\begin{split}
&\sup\limits_{x\in\mathbb{R}}\abs{u(x,t)-1}\to
0~~as~~t\to\infty,\\
&\norm{v}_{L^p(\R)}\to 0
~~\text{as}~~t\to\infty,\ \ 2<p<\infty.
\end{split}
\end{equation}
\end{theorem}
\begin{remark}
\em{The above results hold true regardless of the the amplitude of the initial data. The initial conditions \eqref{pw6-ict1} imply that $(u_0,v_0)$ could be discontinuous, which will bring various difficulties in the analysis. An example of the initial data is the piecewise constant function with arbitrarily large jump discontinuities.}
\end{remark}

\begin{remark}\label{remtech}
\em{Theorem \ref{pw6-th} will be proved by constructing weak solutions as limits of smooth solutions.
Specifically, we first mollify (smoothen) the initial data to obtain the global smooth solutions $(u^\delta, v^\delta)$ and then pass to the limit as $\delta\rightarrow 0$. Compared to the previous works \cite{Lipz,zhangts} for continuous initial data, the main difficulty in the proof is to derive the global {\it a priori} estimates independent of the mollifying parameter $\delta$. In this paper, we shall employ the brilliant idea of Hoff \cite{Hoff-jde, Hoff-arma}, to introduce a time weight function $\sigma=\sigma(t)=\min\{1,t\}$ and the ``effective viscous flux" technique to obtain the desired uniform-in-$\delta$ estimates. The second main difficulty is to obtain the large-time behavior of $v$. Due to the hyperbolicity of
the second equation and low regularity of initial value $v_0$, the regularity of $v$ is hard to attain and the routine energy estimates cannot give the
$L^2$-bound of $v_x$. As a compromise, we succeed in deriving a new estimate for $v$ in the space $v\in L^\infty([0,\infty); L^6(\mathbb{R}))\cap L^6([0,\infty); L^6(\mathbb{R}))$ (see Lemma \ref{pw6-v66}) and obtain the long-time behavior of $v$ as asserted in \eqref{pw6-lb} by making use of the peculiar structure of $\eqref{hp}$. This seems the optimal convergence result we can have for $v$ though the $L^\infty$-convergence is not obtained.}
\end{remark}

\subsection{Stability of viscous shock waves} If $u_-\neq u_+, v_-\neq v_+$,  the existence of (viscous) shock wave can be established (see \cite{jin13}).  The traveling wave solution of $\eqref{hp}$ on $\R$ is a non-constant special solution $(U,V) \in C^{\infty}(\R)$
in the form of
\begin{equation*}
(u,v)(x,t)=(U,V)(z),\ z=x-st, \ s>0,
\end{equation*}
which satisfies
\begin{equation}\label{traveling wave equation}
\begin{cases}
-sU'-\chi(UV)'=DU'',\\
-sV'-U'=0,
\end{cases}
\end{equation}
with boundary condition
\begin{equation*}\label{boundary condition}
U(\pm\infty)=u_\pm,~V(\pm\infty)=v_{\pm},
\end{equation*}
where $'=\frac{d}{dz}$ and $s$ is the wave speed. Here we require $u_\pm> 0$ due to the biological interest. Integrating
(\ref{traveling wave equation}) in $z$ over $\R$ yields the
Rankine-Hugoniot condition as follows
\begin{equation}\label{R-H condition}
\begin{cases}
-s(u_+-u_-)-\chi(u_+v_+-u_-v_-)=0,\\
-s(v_+-v_-)-(u_+-u_-)=0,
\end{cases}
\end{equation}
which gives
\begin{equation}\label{1-6}
s^2+\chi v_+s-\chi u_-=0.
\end{equation}
Solving (\ref{1-6}) for $s$ yields that
\begin{equation}\label{wave speed}
s=\frac{-\chi v_++\sqrt{(\chi v_+)^2+4\chi u_-}}{2}.
\end{equation}

The traveling wave solution $(U,V)$ can be explicitly solved from \eqref{traveling wave equation} and enjoys the following properties (see details in \cite{jin13}).
\begin{proposition}\label{etw}
Assume that $u_\pm$ and $v_\pm$ satisfy $\eqref{R-H condition}$.
Then the system $\eqref{traveling wave equation}$ admits a unique
(up to a translation) monotone traveling wave solutio $(U,V)(x-st)$
with the wave speed $s$ given by (\ref{wave speed}), which
satisfies
$
 U'<0,\ \ V'>0
$
and
\begin{equation*}
|U'|\le \lambda (u_--u_+), \ \ |V'|\le \frac{\lambda (u_--u_+)}{s},
\end{equation*}
where
$
\lambda=\frac{\chi(u_--u_+)}{Ds}>0.
$
\end{proposition}

\begin{theorem}\label{u_+=0}
Let $u_+>0$ and $(U,V)(x-st)$ be a traveling wave solution
of \eqref{traveling wave equation} obtained in Proposition 2.1. Assume that there exists a constant $x_0$ such
that the initial perturbation from the spatially shifted traveling waves with shift $x_0$ of integral zero, namely
$\phi_0(\infty)=\psi_0(\infty)=0$.
Then there exists a  constant $\varepsilon>0$, such that if
\begin{equation*}
\norm{\phi_0}^2+\norm{\psi_0}^2+\norm{u_0-U}^2+\norm{v_0-V}^2\leq \varepsilon, \ \ v_0-V_0 \in L^\infty, \ \ u_0> 0,
\end{equation*}
where
\begin{equation*}
(\phi_0, \psi_0)(x)=-\int_x^{\infty} (u_0(y)-U(y+x_0),
v_0(y)-V(y+x_0))dy,
\end{equation*}
the Cauchy
problem \eqref{hp}-\eqref{Boundary} has a global
weak solution $(u,v)(x,t)$ satisfying
\begin{equation*}
\begin{split}
u-U \in&   L^\infty([0,\infty); L^2(\mathbb{R}))\cap C((0,\infty);C(\mathbb{R})), \ u-U \in L^2([0,\infty);H^1(\mathbb{R})),\\
v-V \in&  L^\infty([0,\infty); L^2\cap L^\infty)\cap
L^2([0,\infty);L^2).
\end{split}
\end{equation*}
Furthermore, the solution has the following asymptotic stability:
\begin{equation*}
\begin{split}
&\sup\limits_{x\in\mathbb{R}}\abs{u(x,t)-U(x+x_0-st)}\to
0~~as~~t\to\infty,\\
&\sup\limits_{x\in\mathbb{R}} \norm{v(x,t)-V(x+x_0-st)}_{L^p} \to
0~{\rm for~all}~2\le p<\infty~~as~~t\to\infty.
\end{split}
\end{equation*}
\end{theorem}

\begin{remark}
\em{The above nonlinear stability results hold true regardless of the size of the wave strength (i.e., $|u_+-u^-|+|v^+-v^-|$ could be arbitrarily large), and the amplitude of initial perturbations $\norm{u_0-U_0}_{L^\infty}$ and $\norm{v_0-V_0}_{L^\infty}$ can be arbitrarily large, which is an significant improvement of previous works (cf. \cite{jin13, Lij13, Li09}) where $\norm{u_0-U_0}_{L^\infty}$ and $\norm{v_0-V_0}_{L^\infty}$ are required to be small.
}
\end{remark}

\section{Some preliminaries}
We first derive a Gronwall-type inequality which will be essentially used in this paper.
\begin{lemma}\label{le-ine} Let the function $y\in W^{1,1}(0,T)$, $\alpha(t)\geq 0$ for $t\geq0$ and $\alpha(t)\geq \beta>0$ for $t\geq T_1>0$ satisfy
\begin{equation}\label{le-i}
y'(t)+\alpha(t)y(t)\le g(t) \ on \ [0, \infty), \ y(0)=y_0,
\end{equation}
where $\beta$ is a positive constant and  $g\in L^1(0, T_1)\cap L^p(T_1, T)$ for some $p\geq1$, and $T_1\in [0, T]$. Then
\begin{equation*}
\begin{split}
\sup_{0\leq t\leq T}y(t)\le |y_0|+(1+\beta^{-1})(\norm{g}_{L^1(0, T_1)}+\norm{g}_{L^p(T_1, T)}).
\end{split}
\end{equation*}
\end{lemma}
\begin{proof} Let $p'$ denote the conjugate number of $p$.
Multiplying
\eqref{le-i} by $e^{\int_0^t \alpha(\tau) d\tau}$ and integrating the resulting inequality over $(0, t)$ yield that
\begin{equation*}
\begin{split}
e^{\int_0^t \alpha(\tau) d\tau}y(t)=&y_0+\int_0^t e^{\int_0^s \alpha(\tau) d\tau} g(s)ds,
\end{split}
\end{equation*}
which gives
\begin{equation*}
\begin{split}
y(t)=&y_0e^{-\int_0^t \alpha(\tau) d\tau}+e^{-\int_0^t \alpha(\tau) d\tau}\int_0^t e^{\int_0^s \alpha(\tau) d\tau} g(s)ds\\
\le& |y_0|+e^{-\int_0^t \alpha(\tau) d\tau}\left(\int_0^{\min\{T_1,t\}} e^{\int_0^s \alpha(\tau) d\tau} g(s)ds+\int_{\min\{T_1,t\}}^t e^{\int_0^s \alpha(\tau) d\tau} g(s)ds\right)\\
\le& |y_0|+\int_0^{\min\{T_1,t\}} e^{-\int_s^t \alpha(\tau) d\tau} |g(s)|ds+\int_{\min\{T_1,t\}}^t e^{-\int_s^t \alpha(\tau) d\tau} |g(s)|ds\\
\le& |y_0|+\int_0^{\min\{T_1,t\}}|g(s)|ds+\int_{\min\{T_1,t\}}^t e^{-\beta(t-s)} |g(s)|ds\\
\le& |y_0|+\int_0^{\min\{T_1,t\}}|g(s)|ds+\norm{g}_{L^p(\min\{T_1,t\}, t)} \norm{e^{-\beta(t-s)}}_{L^{p'}(T_1, t)}\\
\le& |y_0|+(1+\beta^{-1})(\norm{g}_{L^1(0, T_1)}+\norm{g}_{L^p(T_1, T)}),
\end{split}
\end{equation*}
where in the last inequality we have used the following  fact:
\begin{equation*}
\begin{split}
\norm{e^{-\beta(t-s)}}_{L^r(0,t)}=&\left(\int_0^t |e^{-\beta(t-s)}|^r ds\right)^{\frac{1}{r}}\le e^{-\beta t} (\frac{1}{\beta r}e^{\beta r t}-\frac{1}{\beta r})^{\frac{1}{r}}\\
\le& e^{-\beta t} (\frac{1}{\beta r}e^{\beta r t}-\frac{1}{\beta r})^{\frac{1}{r}}\le \beta^{-1},
\end{split}
\end{equation*}
for all $r\in[1, \infty]$. Thus, the proof of Lemma \ref{le-ine} is completed.
\end{proof}

The well-known Aubin-Lions-Simon Lemma (cf. \cite{Rou}) will be used later. For convenience, we state it below.
\begin{lemma}[Aubin-Lions-Simon lemma]
Let $X_0$, $X$ and $X_1$ be three Banach spaces with $X_0 \subseteq X \subseteq X_1$. Suppose that $X_0$ is compactly embedded in $X$ and that $X$ is continuously embedded in $X_1$. For $1 \leq p, q\leq \infty$, let
$$W=\{f \in L^p([0,T]; X_0) | \partial_t f \in L^q([0,T]; X_1)\}.$$

(i) If $p<\infty$, then the embedding of $W$ into $L^p([0,T]; X)$ is compact (that is $W$ is relatively compact in $L^p([0,T]; X)$);

(ii) If $p=\infty$ and $q>1$, then the embedding of $W$ into $C([0,T]; X)$ is compact.

\end{lemma}

\section{Proof of Theorem $\ref{pw6-th}$ }

In this section, we are interested in the dynamics of \eqref{hp}
for fixed values of $D$ and $\chi$. Hence, for simplicity, we take $D=\chi=1$. Now, we begin the proof of Theorem \ref{pw6-th}, by
constructing approximate solutions based upon the mollified initial data. First, we mollify the (coarse) initial data $(u_0, {v}_0)$ as follows:
\begin{equation*}
\begin{split}
u^\delta_0=j^\delta*u_0, \ \ \ v_0^\delta=j^\delta*v_0,
\end{split}
\end{equation*}
where $j^\delta$ is the standard mollifying kernel of width $\delta$ (e.g. see \cite{Adams}). Then we consider the following approximate system
\begin{equation}\label{pw6-hptr}
\begin{cases}
u^\delta_t-(u^\delta v^\delta)_x=u^\delta_{xx},\\
v^\delta_t-u^\delta_x=0,
\end{cases}
\end{equation}
with smooth initial data $(u^\delta_0, v_0^\delta)$ which satisfies
\begin{equation}\label{pw6ini}
(u^\delta_0-1, {\bf v}^\delta_0)\in H^3.
\end{equation}
Using standard arguments, we can obtain the local existence of solutions to the approximate system \eqref{pw6-hptr} with initial data $(u^\delta_0, v_0^\delta)$ satisfying (\ref{pw6ini}).
Next, we shall show in a sequence of lemmas that these approximate solutions
satisfy some global {\it a priori} estimates, independently of the mollifying parameter $\delta$. This will
allow us to take the $\delta$-limit of the sequence of approximate solutions in order to obtain
the solutions of Theorem \ref{pw6-th}.
\subsection{A Priori Estimates for \eqref{pw6-hptr}}
For the sake of simplicity, in this subsection, we still use $(u,  v)$ to
represent the approximate solution $(u^\delta, v^\delta)$. We start with the entropy estimate of $(u,  v)$.

\begin{lemma}\label{pw6-le1} Let $(u, v)$ be a smooth solution of \eqref{hp}-\eqref{Boundary} under the conditions of Theorem $\ref{pw6-th}$. Then there exists a positive constant $C$ independent of $t$ and $\delta$, such that
\begin{equation}\label{pw6-bL^2}
\int_{\mathbb{R}} (u\ln u-u+1) dx+\norm{v}^2+\int_0^T\int_{\mathbb{R}}
\frac{(u_x)^2}{u}dxdt\leq C.
\end{equation}
\end{lemma}
\begin{proof}
\par
Multiplying the first equation of $(\ref{hp})$ by $\ln u$ and the second equation of $(\ref{hp})$ by $v$,
adding the results and integrating the result by parts over $[0,t]\times\R$, we have
\begin{equation*}
\begin{split}
&\int_{\mathbb{R}}(u\ln u-u+1) dx+\int_0^T\int_{\mathbb{R}}
\frac{(u_x)^2}{u}dxdt\\=&\int_{\mathbb{R}}(u_0\ln u_0-u_0+1)dx\le C \int_{\mathbb{R}}(u_0-1)^2dx\le C,
\end{split}
\end{equation*}
which leads to \eqref{pw6-bL^2}. Then, the proof of Lemma \ref{pw6-le1} is completed.
\end{proof}
 To carry out further energy estimates, we introduce a change of $\tilde{u}=u-1$.
Thus, problem \eqref{pw6-hptr} turns
into
\begin{equation}\label{hptr}
\begin{cases}
\tilde{u}_t-\tilde{u}_{xx}=(\tilde{u}v)_x+v_x,\\
v_t-\tilde{u}_x=0,\\
(\tilde{u}, v)(x,0)=(u_0-1, v_0)(x).
\end{cases}
\end{equation}

Next, we will derive the $L^2$ estimate for $(\tilde{u}, v)$. Since our goal is to prove the convergence of the solution to the positive
constant ground state, uniform-in-time estimation of the solution is necessary. Thus, we need a uniform-in-time estimation for $(\tilde{u}, v)$. It
turns out that the standard procedure ($L^2$-type energy estimate) is not sufficient to
achieve our goal, and we need to employ higher-order estimates.

\begin{lemma}\label{pw6le2} Let $(\tilde{u}, v)$ be a smooth solution of \eqref{hptr} under the conditions of Theorem $\ref{pw6-th}$. Then there exists a positive constant $C$ independent of $t$ and $\delta$, such that
\begin{equation}\label{pw6-L^4}
\norm{\tilde{u}}^2+\norm{\tilde{u}}_{L^4}^4+\norm{v}^2+\int_0^T\norm{\tilde{u}_x}^2+\norm{\tilde{uu_x}}^2 dt\leq C.
\end{equation}
\end{lemma}
\begin{proof}
\par The proof of \eqref{pw6-L^4} is due to Li, Li and Zhao \cite{Lipz}.
Multiplying the first equation of $(\ref{hptr})$ by $\tilde{u}$ and the second equation of $(\ref{hptr})$ by $v$, adding the results
and integrating by parts over $\R$, we have
\begin{equation}\label{pw6-1}
\begin{split}
\frac 12\frac{d}{dt}\left(\norm{\tilde{u}}^2+\norm{v}^2\right)+\norm{\tilde{u}_{x}}^2
 =\int_{\R} (\tilde{u}v)_x\tilde{u}dx=-\int_{\R}\tilde{u}v\tilde{u}_{x}dx.
\end{split}
\end{equation}
Multiplying the first equation of $(\ref{hptr})$ by $\tilde{u}^2$
and integrating the result by parts over $\R$, we have
\begin{equation}\label{pw6-2}
\begin{split}
\frac 13\frac{d}{dt}\left(\norm{\tilde{u}}_{L^3}^3\right)+2\int_{\R} \tilde{u}|\tilde{u}_{x}|^2dx
 =-2\int_{\R}\tilde{u}^2v\tilde{u}_{x}dx-2\int_{\R}\tilde{u}v\tilde{u}_{x}dx.
\end{split}
\end{equation}
Multiplying the first equation of $(\ref{hptr})$ by $\tilde{u}^3$
and integrating the result by parts over $\R$, we have
\begin{equation}\label{pw6-3}
\begin{split}
\frac 14\frac{d}{dt}\left(\norm{\tilde{u}}_{L^4}^4\right)+ 3\norm{\tilde{u}|\tilde{u}_{x}|}^2
 =-3\int_{\R}\tilde{u}^3v\tilde{u}_{x}dx-3\int_{\R}\tilde{u}^2v\tilde{u}_{x}dx.
\end{split}
\end{equation}
It follow from $2\times\eqref{pw6-1}-\eqref{pw6-2}$ that
\begin{equation}\label{pw6-4}
\begin{split}
\frac{d}{dt}\left(\norm{\tilde{u}}^2+\norm{v}^2-\frac 13\norm{\tilde{u}}^3\right)+2\norm{\tilde{u}_{x}}^2-2\int_{\R} \tilde{u}|\tilde{u}_{x}|^2dx
 =2\int_{\R}\tilde{u}^2v\tilde{u}_{x}dx.
\end{split}
\end{equation}
The operation $\displaystyle\eqref{pw6-3}+\frac 32\times\eqref{pw6-4}$ leads to
\begin{equation}\label{pw6-5}
\begin{split}
\frac{d}{dt}A(t)+3B(t) =-3\int_{\R}\tilde{u}^3v\tilde{u}_{x}dx.
\end{split}
\end{equation}
Noticing that
\begin{equation}\label{pw6-5-0}
\begin{split}
&A(t)=\frac32\norm{\tilde{u}}^2+\frac32\norm{v}^2-\frac 12\norm{\tilde{u}}_{L^3}^3+\frac14\norm{\tilde{u}}_{L^4}^4
 =\norm{\tilde{u}}^2+\frac18\norm{2\tilde{u}-\tilde{u}^2}^2+\frac18\norm{\tilde{u}}_{L^4}^4+\frac32\norm{v}^2,\\[2mm]
&B(t)=\norm{\tilde{u}_{x}}^2-\int_{\R} \tilde{u}|\tilde{u}_{x}|^2dx+\norm{\tilde{u}|\tilde{u}_{x}|}^2
 =\frac12\norm{\tilde{u}_{x}}^2+\frac12\norm{\tilde{u}_{x}-|\tilde{u}|\tilde{u}_{x}}^2+\frac12\norm{\tilde{u}|\tilde{u}_{x}|}^2.
\end{split}
\end{equation}
For the term on right-hand side of \eqref{pw6-5}, by Cauchy-Schwarz inequality, we get
\begin{equation*}
\begin{split}
-3\int_{\R}\tilde{u}^3v\tilde{u}_{x}dx\le& \frac92\norm{\tilde{u}^2v}^2+\frac12\norm{\tilde{u}|\tilde{u}_{x}|}^2\\
\le& C\norm{\tilde{u}}_{L^\infty}^4\norm{v}_{L^2}^2+\frac12\norm{\tilde{u}|\tilde{u}_{x}|}^2
\le C\norm{\tilde{u}}_{L^\infty}^4+\frac12\norm{\tilde{u}|\tilde{u}_{x}|}^2,
\end{split}
\end{equation*}
which together with \eqref{pw6-5} and \eqref{pw6-5-0} gives
\begin{equation}\label{pw6-7}
\begin{split}
&\frac{d}{dt}\left(\norm{\tilde{u}}^2+\frac18\norm{2\tilde{u}-\tilde{u}^2}^2+\frac18\norm{\tilde{u}}_{L^4}^4+\frac32\norm{v}^2\right)\\
&+\frac32\norm{\tilde{u}_{x}}^2+\frac32\norm{\tilde{u}_{x}-|\tilde{u}|\tilde{u}_{x}}^2+\norm{\tilde{u}|\tilde{u}_{x}|}^2\le
C\norm{\tilde{u}}_{L^\infty}^4.
\end{split}
\end{equation}
Noticing that
\begin{equation*}
\begin{split}
\tilde{u}^4=&4\int_{-\infty}^x \tilde{u}^3\tilde{u}_xdx\le 4\left(\int_{\R}\tilde{u}^6(\tilde{u}+1)dx\right)^{\frac 12} \left(\int_{\R}\frac{\tilde{u}_x^2}{\tilde{u}+1}dx\right)^{\frac 12}\\
\le& 4\norm{\tilde{u}}_{L^\infty}^2\left(\int_{\R}\tilde{u}^2(\tilde{u}+1)dx\right)^{\frac 12} \left(\int_{\R}\frac{\tilde{u}_x^2}{u}dx\right)^{\frac 12},
\end{split}
\end{equation*}
which gives
\begin{equation*}
\begin{split}
\norm{\tilde{u}}_{L^\infty}^4
\le& 4\norm{\tilde{u}}_{L^\infty}^2\left(\int_{\R}\tilde{u}^2(\tilde{u}+1)dx\right)^{\frac 12} \left(\int_{\R}\frac{\tilde{u}_x^2}{u}dx\right)^{\frac 12}\\
\le& \frac12\norm{\tilde{u}}_{L^\infty}^4+8\left(\int_{\R}\tilde{u}^2(\tilde{u}+1)dx\right)\left(\int_{\R}\frac{\tilde{u}_x^2}{u}dx\right)\\
\le& \frac12\norm{\tilde{u}}_{L^\infty}^4+C\left(\int_{\R}\left(\tilde{u}^4+\tilde{u}^2\right)dx\right)\left(\int_{\R}\frac{\tilde{u}_x^2}{u}dx\right).
\end{split}
\end{equation*}
Thus,
\begin{equation*}
\begin{split}
\norm{\tilde{u}}_{L^\infty}^4
\le& C\left(\int_{\R}\left(\tilde{u}^4+\tilde{u}^2\right)dx\right)\left(\int_{\R}\frac{\tilde{u}_x^2}{u}dx\right).
\end{split}
\end{equation*}
Substituting the above inequality into \eqref{pw6-7} and applying Gronwall's inequality, we obtain \eqref{pw6-L^4}.
\end{proof}

Next, we want to derive the appropriate estimates for the first order derivative of $(u,v)$. Since we plan to use the limit of the mollified function $(u^\delta, v^\delta)$ as $\delta \to 0$ to obtain the solution of our target system \eqref{hp}-\eqref{Boundary}, the estimates of the first order derivative of $(u^\delta, v^\delta)$ need to be independent of $\delta$. If we employ the method for $H^1$-estimates used in \cite{GXZZ, Lipz, zhangts}, we shall encounter the term $\int_\R (|u^\delta_{0x}|^2+|v^\delta_{0x}|^2)dx$ which is out of control since the initial assumption of $(u_0, v_0)$ is not yet up to $H^1(\R)$, see (\ref{pw8mde0}). Indeed in general the bound of $\int_\R (|u^\delta_{0x}|^2+|v^\delta_{0x}|^2)dx$ is of order $\frac{1}{\delta}$ given that $L^2(\R)$-norm is bounded (see \cite[Lemma 1.2]{OP}). Hence we have to find an idea to avoid the estimates of first-order derivative of $(u^\delta_0, v^\delta_0)$ to attain the uniform boundedness of first-order estimates in $\delta$. Inspired by the brilliant idea of Hoff \cite{Hoff-1992, Hoff-jde} of treating discontinuous data, we introduce a weight function $\sigma=\sigma(t)=\min\{1, t\}$ to resolve this obstacle. Unfortunately, this method is  not valid to $v^\delta$. Since, in this framework, to avoid the estimate of $\int_\R |v^\delta_{0x}|^2dx$, the uniform-in-$\delta$ bound of $\int_0^T\int_\R |v^\delta_{x}|^2dxdt$ is necessary. It turns out that it is nearly impossible to get the uniform-in-$\delta$ of $\int_0^T\int_\R |v^\delta_{x}|^2dxdt$ when the second equation of \eqref{hp} is hyperbolic (no diffusion term with respect to $v$) and $v$ has only lower-regularity initial data ($v_0 \in L^2(\mathbb{R})\cap L^\infty(\mathbb{R})$). Thus, we can only get the first order derivative of $u^\delta$ in the following.


\begin{lemma}\label{pw6-h^1 estimate}
Let the conditions of Theorem $\ref{pw6-th}$ hold and $(\tilde{u}, v)$ be a smooth solution of \eqref{hptr}. Then there exists a positive constant $C$ independent of $t$  and $\delta$, such that
\begin{equation}\label{pw6-L^2}
\begin{split}
\sigma\norm{\tilde{u}_x}^2+\sigma^2\norm{\tilde{u}_t}^2+\sigma^2\norm{v_t}^2+\int_0^T\sigma\norm{\tilde{u}_t}^2dt+\int_0^T\sigma^2\norm{\tilde{u}_{xt}}^2dt\le
C,
\end{split}
\end{equation}
where  $\sigma=\sigma(t)=\min\{1,t\}$.
\end{lemma}
\begin{proof}
\par
We first multiply the first equation of $(\ref{hptr})$ by $\sigma \tilde{u}_t$ and
integrate the resulting equation over $\R\times[0, T]$ to get
\begin{equation}\label{pw6-es-ut}
\begin{split}
&\frac{1}{2}\sigma\norm{\tilde{u}_x}^2+\int_0^T\sigma\norm{\tilde{u}_t}^2dt\\
=&\frac{1}{2}\int_0^{\sigma(t)}\norm{\tilde{u}_x}^2dt-\int_0^T\sigma\int_{\mathbb{R}}
(\tilde{u}v)\tilde{u}_{xt} dxdt-\int_0^T\sigma\int_{\mathbb{R}}v\tilde{u}_{xt}dxdt .
\end{split}
\end{equation}
For the first term on the right-hand side of \eqref{pw6-es-ut}, we have from \eqref{pw6-L^4}
\begin{equation}\label{pw6-es1}
\begin{split}
\frac{1}{2}\int_0^{\sigma(t)}\norm{\tilde{u}_x}^2dt\le \frac{1}{2}\int_0^T\norm{\tilde{u}_x}^2dt \le C.
\end{split}
\end{equation}
For the second term on the right-hand side of \eqref{pw6-es-ut}, we have
\begin{equation}\label{pw6-es2}
\begin{split}
-\int_0^T\sigma\int_{\mathbb{R}}
(\tilde{u}v)\tilde{u}_{xt} dxdt=&-\int_0^T\left(\sigma\int_{\mathbb{R}}
\tilde{u}v\tilde{u}_{x} dx\right)_tdt+\int_0^{\sigma(t)}\int_{\mathbb{R}}
\tilde{u}v\tilde{u}_{x} dxdt\\
&+\int_0^T\sigma\int_{\mathbb{R}}
\tilde{u}_tv\tilde{u}_{x} dxdt+\int_0^T\sigma\int_{\mathbb{R}}
\tilde{u}v_t\tilde{u}_{x} dxdt\\[2mm]
=& H_1+H_2+H_3+H_4.
\end{split}
\end{equation}
By the Sobolev inequality $\norm{f}_{L^\infty}^2\le 2\norm{f}\norm{f_x}$, Cauchy-Schwarz inequality and \eqref{pw6-L^4}, we have
\begin{equation*}
\begin{split}
H_1=&-\sigma\int_{\mathbb{R}}
\tilde{u}v\tilde{u}_{x} dx\le \frac{\sigma}{8}\norm{\tilde{u}_{x}}^2+2\sigma\norm{\tilde{u}v}^2\\[2mm]
\le& \frac{\sigma}{8}\norm{\tilde{u}_{x}}^2+2\sigma\norm{\tilde{u}}^2_{L^\infty}\norm{v}^2\\[2mm]
\le& \frac{\sigma}{8}\norm{\tilde{u}_{x}}^2+C\sigma\norm{\tilde{u}}\norm{\tilde{u}_x}\\[2mm]
\le& \frac{\sigma}{4}\norm{\tilde{u}_{x}}^2+C\sigma\norm{\tilde{u}}^2\\[2mm]
\le& \frac{\sigma}{4}\norm{\tilde{u}_{x}}^2+C.
\end{split}
\end{equation*}
Employing Cauchy-Schwarz inequality and \eqref{pw6-L^4} again, we can estimate $H_2$ as
\begin{equation*}
\begin{split}
H_2=&\int_0^{\sigma(t)}\int_{\mathbb{R}}
\tilde{u}v\tilde{u}_{x} dxdt
\le  \frac{1}{4}\int_0^{\sigma(t)}\norm{\tilde{u}_{x}}^2dt+C\int_0^{\sigma(t)}\norm{\tilde{u}}^2dt\\
\le& \frac{1}{4}\int_0^{t}\norm{\tilde{u}_{x}}^2dt+C\norm{\tilde{u}}^2\le C,
\end{split}
\end{equation*}
where $\sigma(t)=\min\{1,t\}$ has been used. By the Sobolev
inequality, Cauchy-Schwarz inequality and \eqref{pw6-L^4}, we have
\begin{equation*}
\begin{split}
H_3=&\int_0^T\sigma\int_{\mathbb{R}}
\tilde{u}_tv\tilde{u}_{x} dxdt\\
\le&  C\int_0^{t}\norm{\tilde{u}_{x}}^2dt+\eta\int_0^{t}\sigma^2\norm{\tilde{u}_t v}^2dt\\
\le& C\int_0^{t}\norm{\tilde{u}_{x}}^2dt+\eta\int_0^{t}\sigma^2\norm{\tilde{u}_t}_{L^\infty}^2\norm{v}^2dt\\
\le& C\int_0^{t}\norm{\tilde{u}_{x}}^2dt+\eta\int_0^{t}\sigma^2\norm{\tilde{u}_t}\norm{\tilde{u}_{xt}}\norm{v}^2dt\\
\le& C+C\eta\int_0^{t}\sigma^2\norm{\tilde{u}_t}\norm{\tilde{u}_{xt}}dt\\
\le& C+\frac{1}{2}\int_0^{t}\sigma^2\norm{\tilde{u}_t}^2+C\eta\int_0^{t}\sigma^2\norm{\tilde{u}_{xt}}^2dt,
\end{split}
\end{equation*}
where $\eta>0$ is a positive constant which will be determined later. For $H_4$, we have from the fact $v_t=\tilde{u}_x$ and Cauchy-Schwarz inequality that
\begin{equation*}
\begin{split}
H_4=&\int_0^T\sigma\int_{\mathbb{R}}
\tilde{u}v_t\tilde{u}_{x} dxdt =\int_0^T\sigma\int_{\mathbb{R}}
\tilde{u}|\tilde{u}_{x}|^2 dxdt \\
\le& \int_0^{t}\norm{\tilde{u}_{x}}^2dt+\int_0^{t}\norm{\tilde{u}\tilde{u}_{x}}^2dt\le C,
\end{split}
\end{equation*}
where we have used \eqref{pw6-L^4}. Substituting the estimates of $H_1-H_4$ into \eqref{pw6-es2} to get
\begin{equation}\label{pw6-es-2}
\begin{split}
-\int_0^T\sigma\int_{\mathbb{R}}
(\tilde{u}v)\tilde{u}_{xt} dxdt\le \frac{\sigma}{4}\norm{\tilde{u}_{x}}^2+\frac{1}{2}\int_0^{t}\sigma^2\norm{\tilde{u}_t}^2+C\eta\int_0^{t}\sigma^2\norm{\tilde{u}_{xt}}^2dt+C.
\end{split}
\end{equation}
For the last term on the right-hand side of \eqref{pw6-es-ut}, we have
\begin{equation}\label{pw6-es23}
\begin{split}
-\int_0^T\sigma\int_{\mathbb{R}}
v\tilde{u}_{xt} dxdt=&-\int_0^T\left(\sigma\int_{\mathbb{R}}
v\tilde{u}_{x} dx\right)_tdt+\int_0^{\sigma(t)}\int_{\mathbb{R}}
v\tilde{u}_{x} dxdt\\
&+\int_0^T\sigma\int_{\mathbb{R}}
v_t\tilde{u}_{x} dxdt\\
=& I_1+I_2+I_3.
\end{split}
\end{equation}
By Cauchy-Schwarz inequality and \eqref{pw6-L^4}, we can estimate $I_1-I_3$ as
\begin{equation*}
\begin{split}
I_1=&-\sigma\int_{\mathbb{R}}
v\tilde{u}_{x} dx\le \frac{\sigma}{8}\norm{\tilde{u}_{x}}^2+2\sigma\norm{v}^2\le\frac{\sigma}{8}\norm{\tilde{u}_{x}}^2+C,
\end{split}
\end{equation*}

\begin{equation*}
\begin{split}
I_2=\int_0^{\sigma(t)}\int_{\mathbb{R}}
v\tilde{u}_{x} dxdt
&\le  \int_0^{\sigma(t)}\norm{\tilde{u}_{x}}^2dt+C\int_0^{\sigma(t)}\norm{\tilde{u}}^2dt\\
&\le \int_0^{t}\norm{\tilde{u}_{x}}^2dt+C\le C
\end{split}
\end{equation*}
and
\begin{equation*}
\begin{split}
I_3=&\int_0^T\sigma\int_{\mathbb{R}}
v_t\tilde{u}_{x} dxdt =\int_0^T\sigma\int_{\mathbb{R}}
|\tilde{u}_{x}|^2 dxdt\le C.
\end{split}
\end{equation*}
Substituting the estimates of $I_1-I_3$ into \eqref{pw6-es23}, we have
\begin{equation}\label{pw6-es-23}
\begin{split}
-\int_0^T\sigma\int_{\mathbb{R}}
v\tilde{u}_{xt} dxdt\le \frac{\sigma}{8}\norm{\tilde{u}_{x}}^2+C.
\end{split}
\end{equation}
Putting \eqref{pw6-es1}, \eqref{pw6-es-2} and \eqref{pw6-es-23} into \eqref{pw6-es-ut}, we conclude that
\begin{equation}\label{pw6-es6}
\begin{split}
\sigma\norm{\tilde{u}_x}^2+\int_0^T\sigma\norm{\tilde{u}_t}^2dt\le C+C\eta\int_0^{t}\sigma^2\norm{\tilde{u}_{xt}}^2dt.
\end{split}
\end{equation}

Next, in order to obtain the estimate of
$\displaystyle\int_0^T\sigma^2\norm{u_{xt}}^2dt$, differentiating $\eqref{hptr}$ with respect time $t$, we get
\begin{equation}\label{hptr2}
\begin{cases}
\tilde{u}_{tt}-\tilde{u}_{xxt}=(\tilde{u}v)_{xt}+v_{xt},\\
v_{tt}-\tilde{u}_{xt}=0.
\end{cases}
\end{equation}
Multiplying the first equation of $\eqref{hptr2}$ by $\sigma^2\tilde{u}_t$ and the
second by $\sigma^2v_t$, adding the results and integrating it over $\R\times[0, T]$, we have
\begin{equation}\label{pw6-es-utt}
\begin{split}
&\frac{\sigma^2}{2}\norm{\tilde{u}_t}^2+\frac{\sigma^2}{2}\norm{v_t}^2+\int_0^T\sigma^2\norm{\tilde{u}_{xt}}^2dt\\
\le&\int_0^T\sigma\norm{\tilde{u}_t}^2dt+\int_0^T\sigma\norm{v_t}^2dt-\int_0^T\sigma^2\int_{\mathbb{R}}(\tilde{u}v)_t \tilde{u}_{xt}dxdt\\
\le &\int_0^T\sigma\norm{\tilde{u}_t}^2dt+\int_0^T\sigma\norm{\tilde{v}_t}^2dt-\int_0^T\sigma^2\int_{\mathbb{R}}\tilde{u}v_t \tilde{u}_{xt}dxdt-\int_0^T\sigma^2\int_{\mathbb{R}}\tilde{u}_t v \tilde{u}_{xt}dxdt\\
\le &\int_0^T\sigma\norm{\tilde{u}_t}^2dt+\int_0^T\sigma\norm{\tilde{u}_x}^2dt-\int_0^T\sigma^2\int_{\mathbb{R}}\tilde{u}\tilde{u}_x \tilde{u}_{xt}dxdt-\int_0^T\sigma^2\int_{\mathbb{R}}\tilde{u}_t v \tilde{u}_{xt}dxdt\\
=& J_1+J_2+J_3+J_4,
\end{split}
\end{equation}
where we have used the integration by parts and $v_t=\tilde{u}_x$. From \eqref{pw6-L^4} and \eqref{pw6-es6}, we have
\begin{equation*}
\begin{split}
J_1+J_2\le C+C\eta\int_0^{t}\sigma^2\norm{\tilde{u}_{xt}}^2dt.
\end{split}
\end{equation*}
For the $J_3$, one has from Cauchy-Schwarz inequality and  \eqref{pw6-L^4} that
\begin{equation*}
\begin{split}
J_3\le  \frac{1}{4}\int_0^T\sigma^2 \norm{\tilde{u}_{xt}}^2 dt+ \int_0^T\sigma^2\norm{\tilde{u}\tilde{u}_x}^2dt\le C+\frac{1}{4}\int_0^T\sigma^2 \norm{\tilde{u}_{xt}}^2 dt.
\end{split}
\end{equation*}
For $J_4$, by the Sobolev
inequality, Cauchy-Schwarz inequality and \eqref{pw6-L^4}, we have
\begin{equation*}
\begin{split}
J_4\le& \frac{1}{4}\int_0^T\sigma^2 \norm{\tilde{u}_{xt}}^2 dt+ \int_0^T\sigma^2\norm{\tilde{u}_t v}^2dt\\[2mm]
\le& \frac{1}{4}\int_0^T\sigma^2 \norm{\tilde{u}_{xt}}^2 dt+ C\int_0^T\sigma^2\norm{\tilde{u}_t}_{L^\infty}^2\norm{v}^2dt\\[2mm]
\le& \frac{1}{4}\int_0^T\sigma^2 \norm{\tilde{u}_{xt}}^2 dt+ C\int_0^T\sigma^2\norm{\tilde{u}_t}\norm{\tilde{u}_{xt}}\norm{v}^2dt\\[2mm]
\le& \frac{1}{2}\int_0^T\sigma^2 \norm{\tilde{u}_{xt}}^2 dt+ C\int_0^T\sigma^2\norm{\tilde{u}_t}^2dt.
\end{split}
\end{equation*}
Substituting the estimates of $J_1-J_4$ into \eqref{pw6-es-utt}, we get
\begin{equation}\label{pw6-es9}
\begin{split}
\sigma^2\norm{\tilde{u}_t}^2+\sigma^2\norm{v_t}^2+\int_0^T\sigma^2\norm{\tilde{u}_{xt}}^2dt
\le C+C\int_0^T\sigma\norm{\tilde{u}_t}^2dt,
\end{split}
\end{equation}
which updates \eqref{pw6-es6} as
\begin{equation*}
\begin{split}
\sigma\norm{\tilde{u}_x}^2+\int_0^T\sigma\norm{\tilde{u}_t}^2dt\le C+C\eta\int_0^T\sigma\norm{\tilde{u}_t}^2dt.
\end{split}
\end{equation*}
By choosing $\eta$ sufficiently small, we get from the above inequality that
\begin{equation}\label{pw6-es10}
\begin{split}
\sigma\norm{\tilde{u}_x}^2+\int_0^T\sigma\norm{\tilde{u}_t}^2dt\le C,
\end{split}
\end{equation}
which together with \eqref{pw6-es9} yields
\begin{equation*}\label{pw6-es11}
\begin{split}
\sigma^2\norm{\tilde{u}_t}^2+\sigma^2\norm{v_t}^2+\int_0^T\sigma^2\norm{\tilde{u}_{xt}}^2dt
\le C.
\end{split}
\end{equation*}
This, along with \eqref{pw6-es10}, leads immediately to
\eqref{pw6-L^2}. Thus, the proof of Lemma \ref{pw6-h^1 estimate} is completed.
\end{proof}

Now, we can deduce the large-time behavior of $\tilde{u}$.
\begin{lemma}\label{lnn}
Let the conditions of Theorem $\ref{pw6-th}$ hold and let $(\tilde{u}, v)$ be a smooth solution of \eqref{hptr}. Then
it follows that
\begin{equation*}
\begin{split}
\sup\limits_{x\in\mathbb{R}}\abs{\tilde{u}(x,t)}\to
0~~as~~t\to\infty.
\end{split}
\end{equation*}
\end{lemma}
\begin{proof}
\par
From \eqref{pw6-L^4} and \eqref{pw6-L^2}, by the fact $\sigma=1$ for $t\geq1$, we have
\begin{equation*}
\begin{split}
\int_1^\infty\norm{\tilde{u}_{x}}^2dt+\int_1^\infty\norm{\tilde{u}_{xt}}^2dt\le C,
\end{split}
\end{equation*}
which implies that
\begin{equation*}
\norm{\tilde{u}_{x}(\cdot,t)}^2\to 0
~~\text{as}~~t\to\infty.
\end{equation*}
Hence,
\begin{equation*}
\begin{split}
\tilde{u}_{x}^2(x,t)&=2\left|\int^{\infty}_x\tilde{u}\tilde{u}_{x}(y,t)dy\right|\\
             &\leq 2\left(\int_{-\infty}^\infty\tilde{u}^2dy\right)^{1/2}\left(\int_{-\infty}^\infty|\tilde{u}_{x}|^2dy\right)^{1/2}\\
             &\leq C\norm{\tilde{u}_{x}(\cdot,t)}\to 0 ~~\mathrm{as}~~ t\to\infty,
\end{split}
\end{equation*}
which completes the proof.
\end{proof}

\begin{remark}\label{lrem}
It follow from  Lemma \ref{lnn} that  there exists a positive constant $\hat{T}>1$, such that $|\tilde{u}(x,t)|<\frac{1}{2}$ for any $t\geq\hat{T}$.
\end{remark}
We now proceed to derive a uniform (in time) upper bound for $v$.  Motivated by the studies for the Navier-Stokes equations (cf. \cite{Hoff-1992, Hoff-jde, Hoff-arma}), we here introduce the following so-called ``effective viscous flux $F(x,t)$'':
\begin{equation}\label{pw6-evf}
F=\tilde{u}_x+(\tilde{u}+1)v.
\end{equation}
From the first equation of $\eqref{hptr}$, it is easy to see that
\begin{equation}\label{pw6-fut}
F_x=\tilde{u}_t.
\end{equation}

\begin{lemma}\label{pw6-secl}
Assume the conditions of Theorem $\ref{pw6-th}$ hold. Let $(\tilde{u}, v)$ be a smooth solution of \eqref{hptr}. Then there exists a positive constant $C$ independent of $t$  and $\delta$, such that
\begin{equation}\label{pw6-vin}
\norm{v}_{L^\infty} \le C.
\end{equation}
\end{lemma}
\begin{proof}
It follows from $v_t=\tilde{u}$ that
\begin{equation}\label{pw6-vc}
\begin{split}
v_t+(\tilde{u}+1)v=F.
\end{split}\end{equation}
Taking
\begin{equation*}
\begin{split}
y=v, \ \ \alpha(t)=\tilde{u}+1, \ \ T_1=\hat{T}>1,\ \ \beta=\frac12,\ \
g(t)=F,\ \ p=4
\end{split}\end{equation*}
in Lemma \ref{le-ine}, we have
\begin{equation}\label{pw6-in}
\begin{split}
v\le \norm{v_0}_{L^\infty}+3(\norm{g}_{L^1(0,\hat{T})}+\norm{g}_{L^4(\hat{T},t)}).
\end{split}\end{equation}
For $0\leq t\leq\hat{T}$, we have from the H\"older and Sobolev inequalities
\begin{equation}\label{pw6-f1}
\begin{split}
\int_0^{\hat{T}}\norm{F}_{L^\infty}dt\le& \sqrt{2}\int_0^{\hat{T}}\norm{F}^\frac{1}{2}\norm{F_x}^\frac{1}{2}dt\\
\le& \sqrt{2}\left(\int_0^{\hat{T}}\norm{F}^2dt\right)^\frac{1}{4}\left(\int_0^{\hat{T}}\sigma\norm{F_x}^2dt\right)^\frac{1}{4}
\left(\int_0^{\hat{T}}\sigma^{-\frac{1}{2}}dt\right)^\frac{1}{2}\\
\le& C\left(\int_0^{\hat{T}}\norm{F}^2dt\right)^\frac{1}{4}\left(\int_0^{\hat{T}}\sigma\norm{F_x}^2dt\right)^\frac{1}{4}\left(\int_0^{1}t^{-\frac{1}{2}}dt+\hat{T}-1\right)^\frac{1}{2}\\
\le& C\left(\int_0^{\hat{T}}\norm{F}^2dt\right)^\frac{1}{4}\left(\int_0^{\hat{T}}\sigma\norm{F_x}^2dt\right)^\frac{1}{4}.
\end{split}\end{equation}
Using \eqref{pw6-evf}, \eqref{pw6-L^4}, Sobolev and Cauchy-Schwarz inequalities, we have
\begin{equation*}
\begin{split}
\int_0^{\hat{T}}\norm{F}^2dt\le& C\int_0^{\hat{T}} \left(\norm{\tilde{u}_x}^2+\norm{\tilde{u}v}^2+\norm{v}^2\right)dt\\
\le& C\int_0^{\hat{T}} \left(\norm{\tilde{u}_x}^2+\norm{\tilde{u}}\norm{\tilde{u}_x}\norm{v}^2+\norm{v}^2\right)dt\\
\le& C\int_0^{\hat{T}} \left(\norm{\tilde{u}_x}^2+1\right)dt\le C.\\
\end{split}\end{equation*}
By \eqref{pw6-fut} and \eqref{pw6-L^2}, we get
\begin{equation*}
\begin{split}
\int_0^{\hat{T}}\sigma\norm{F_x}^2dt\le& \int_0^{\hat{T}}\sigma\norm{\tilde{u}_t}^2dt\le C.
\end{split}\end{equation*}
The above two inequalities update \eqref{pw6-f1} as
\begin{equation}\label{pw6-f1f}
\begin{split}
\int_0^{\hat{T}}\norm{F}_{L^\infty}dt\le C.
\end{split}
\end{equation}
For $\hat{T}\le t\le T$, one deduces from Sobolev inequality and \eqref{pw6-fut} that
\begin{equation}\label{pw6-f2}
\begin{split}
\int_{\hat{T}}^t\norm{F}_{L^\infty}^4dt\le& C\int_{\hat{T}}^t\norm{F}^2\norm{F_x}^2dt\\
\le& C\left(\sup_{t\geq \hat{T}}\norm{F}^2\right)\int_{\hat{T}}^t\norm{F_x}^2dt
\le C\left(\sup_{t\geq \hat{T}}\norm{F}^2\right)\int_{\hat{T}}^t\sigma\norm{\tilde{u}_t}^2dt\\
\le&C\left(\sup_{t\geq \hat{T}}\norm{F}^2\right),
\end{split}\end{equation}
where we have used the fact that $\sigma(t)=1$ for  $t\geq \hat{T}>1$. Using \eqref{pw6-evf}, Sobolev and Cauchy-Schwarz inequalities, \eqref{pw6-L^4} and \eqref{pw6-L^2}, we have
\begin{equation}\label{pw6-f2-0}
\begin{split}
\sup_{t\geq \hat{T}}\norm{F}^2 \le& C\sup_{t\geq \hat{T}}\left(\norm{\tilde{u}_x}^2+\norm{\tilde{u}v}^2+\norm{v}^2\right)\\
\le& C\sup_{t\geq \hat{T}} \left(\norm{\tilde{u}_x}^2+\norm{\tilde{u}}\norm{\tilde{u}_x}\norm{v}^2+\norm{v}^2\right)\\
\le& C\sup_{t\geq \hat{T}} \left(\sigma\norm{\tilde{u}_x}^2+1\right)\le C,
\end{split}\end{equation}
where in the last inequality we have used $\sigma(t)=1$ for  $t\geq \hat{T}$ again. This together with \eqref{pw6-f2} gives
$
\int_{\hat{T}}^t\norm{F}_{L^\infty}^4dt\le C,
$
which along with  \eqref{pw6-in} and \eqref{pw6-f1f} gives \eqref{pw6-vin}.
\end{proof}

\begin{lemma}\label{pw6-v66} Let the assumptions in Theorem $\ref{pw6-th}$ hold. Then there exists a positive constant $C$ independent of $t$  and $\delta$, such that
\begin{equation}\label{pw6-vcf}
 \sup\limits_{t\in[0,T]}\int_{\R}v^{6}dx+\int_{0}^t\int_{\R}v^{6}dxdt\le C.
\end{equation}
\end{lemma}
\begin{proof}
Multiplying \eqref{pw6-vc} by $v^{5}$ and integrating the resulting equality over $\R$, one has
\begin{equation}\label{pw6-vc1}
\begin{split}
\frac1{6} \left(\int_{\R}v^{6}dx\right)_t+\int_{\R}(\tilde{u}+1)v^{6}dx=\int_{\R}Fv^{5}dx.
\end{split}\end{equation}
Integrating the above equality over $[\hat{T}, t)$, we have
\begin{equation}\label{pw6-vc2}
\begin{split}
\frac1{6}\int_{\R}v^{6}dx+\frac12\int_{\hat{T}}^t\int_{\R}v^{6}dxdt\le \sup\limits_{t\in[0,\hat{T}]}\left(\frac1{6}\int_{\R}v^{6}dx\right)+ \int_{\hat{T}}^t\int_{\R}|F||v|^{5}dxdt,
\end{split}\end{equation}
where we have used Remark \ref{lrem}. We need to further estimate the last term in \eqref{pw6-vc2}. By the Young
inequality, we have that
\begin{equation}\label{pw6-young}
\begin{split}
\int_{\hat{T}}^t\int_{\R}|F||v|^{5}dxdt\le \frac 14\int_{\hat{T}}^t\int_{\R}v^{6}dxdt+C\int_{\hat{T}}^t\int_{\R}|F|^{6}dxdt,
\end{split}\end{equation}
which updates \eqref{pw6-vc2} as
\begin{equation}\label{pw6-vc3}
\begin{split}
\frac1{6}\int_{\R}v^{6}dx+\frac14\int_{\hat{T}}^t\int_{\R}v^{6}dxdt\le \sup\limits_{t\in[0,\hat{T}]}\left(\frac1{6}\int_{\R}v^{6}dx\right)+ C\int_{\hat{T}}^t\int_{\R}|F|^{6}dxdt.
\end{split}\end{equation}
For the first term on the right hand side of \eqref{pw6-vc3}, we have from \eqref{pw6-L^4} and \eqref{pw6-vin} that
\begin{equation*}
\begin{split}
\sup\limits_{t\in[0,\hat{T}]}\left(\frac1{6}\int_{\R}v^{6}dx\right)\le C\norm{v}_{L^\infty}^4\norm{v}^2\le C.
\end{split}\end{equation*}
It follows from Gagliardo-Nirenberg inequality, \eqref{pw6-fut} and \eqref{pw6-f2-0} that
\begin{equation}\label{pw6-F6}
\begin{split}
\int_{\hat{T}}^t\norm{F}_{L^6}^6dt\le& \int_{\hat{T}}^t\norm{F}^4\norm{F_x}^2dt
\le\left(\sup_{t\geq \hat{T}}\norm{F}^4\right)\int_{\hat{T}}^t\norm{F_x}^2dt\\
\le&\left(\sup_{t\geq \hat{T}}\norm{F}^4\right)\int_{\hat{T}}^t\sigma\norm{\tilde{u}_t}^2dt
\le C,
\end{split}\end{equation}
where we have used the fact that $\sigma(t)=1$ for  $t\geq \hat{T}>1$. Substituting the above two inequalities into \eqref{pw6-vc3}, we obtain
\begin{equation*}
\begin{split}
 \sup\limits_{t\in[\hat{T},\infty]}\int_{\R}v^{6}dx+\int_{\hat{T}}^t\int_{\R}v^{6}dxdt\le C.
\end{split}\end{equation*}
For $0\leq t\leq\hat{T}$, we have from \eqref{pw6-L^4} and \eqref{pw6-vin} that
\begin{equation*}
\begin{split}
 &\sup\limits_{t\in[0,\hat{T}]}\int_{\R}v^{6}dx+\int_0^{\hat{T}}\int_{\R}v^{6}dxdt\\
 \le& \sup\limits_{t\in[0,\hat{T}]}\norm{v}_{L^\infty}^4\int_{\R}v^{2}dx+\hat{T}\sup\limits_{t\in[0,\hat{T}]}\norm{v}_{L^\infty}^4\int_{\R}v^{2}dx\le C.
\end{split}\end{equation*}
By coupling the above two inequalities together yields \eqref{pw6-vcf} and completes the proof.
\end{proof}
\subsection{Proof of Theorem $\ref{pw6-th}$}
We now prove Theorem \ref{pw6-th}. It first follows from  Lemmas \ref{pw6le2}-\ref{pw6-v66} that
\begin{equation}\label{pw6-le-last}
\begin{cases}
\begin{split}
&\norm{u^\delta-1}^2+\norm{u^\delta-1}_{L^4}^4+\norm{v^\delta}^2+\int_0^T\norm{u^\delta_x}^2+\norm{(u^\delta-1)u^\delta_x}^2 dt\leq C,\\
&\sigma\norm{u^\delta_x}^2+\sigma^2\norm{u^\delta_t}^2+\sigma^2\norm{v^\delta_t}^2+\int_0^T\sigma\norm{u^\delta_t}^2dt+\int_0^T\sigma^2\norm{u^\delta_{xt}}^2dt\le C,\\
&\|v^\delta\|_{L^\infty}+\|v^\delta\|_{L^6}^6+\int_{0}^t\|v^\delta\|_{L^6}^6dt\le C,
\end{split}
\end{cases}
\end{equation}
which gives
\begin{equation}\label{1-le-last}
\begin{cases}
u^\delta-1\in L^\infty([0,\infty), L^2(\mathbb{R})), \ \ \ \ \ (v^\delta_t, \nabla{u}^\delta)\in L^2([0,\infty), L^2(\mathbb{R})), \\[2mm]
u^\delta-1\in L^\infty((0,\infty),W^{1,2}(\mathbb{R})),\ \ \ u^\delta_t\in L^2((0,\infty),H^1(\mathbb{R})),\\[2mm]
v^\delta\in L^\infty([0,\infty); L^2(\mathbb{R})\cap L^{\infty}(\mathbb{R}))\cap L^6([0,\infty); L^6(\mathbb{R})).
\end{cases}
\end{equation}
By \eqref{1-le-last} and the Aubin-Lions-Simon lemma, we can extract a subsequence, still denoted by $(u^\delta, v^\delta)$, such that the following convergence hold as $\delta\to 0$
\begin{equation*}
\begin{cases}
v^\delta(\cdot,t)\to {\bf v}\ \text{strongly  in} \ C([0,\infty), H^{-1}(\mathbb{R})), \\[2mm]
u^\delta(\cdot,t)\to u(\cdot,t)\ \text{strongly  in} \ C((0,\infty), C(\mathbb{R})),\\[2mm]
u_x^\delta(\cdot,t)\to u_x(\cdot,t)\ \text{weakly  in} \ L^2([0,\infty),L^2(\mathbb{R})).
\end{cases}
\end{equation*}
Thus, it is easy to show that the limit function $(u, v)$  is indeed a weak solution of the system  \eqref{hp}-\eqref{Boundary} and inherits all the bounds of \eqref{pw6-le-last}. Thus, \eqref{pw6-le-last0} is proved.

To complete the proof of  Theorem $\ref{pw6-th}$,
we only need to prove $\eqref{pw6-lb}$. On the other hand, by \eqref{pw6-vc1}, we have
\begin{equation*}
\begin{split}
\int_{\hat{T}}^t\abs{\left(\norm{v}_{L^6}^6\right)_t}dt\le& C\int_{\hat{T}}^t\int_{\R}(\tilde{u}+1)v^{6}dxdt+C\int_{\hat{T}}^t\int_{\R}Fv^{5}dxdt\\
\le& C\int_{\hat{T}}^t\int_{\R}v^{6}dxdt+C\int_{\hat{T}}^t\int_{\R}|F||v|^{5}dxdt,
\end{split}\end{equation*}
where the boundedness of $\tilde{u}$ for $t>\hat{T}$ has been used (see Remark \ref{lrem}). This together with \eqref{pw6-young}, \eqref{pw6-F6} and \eqref{pw6-vcf} implies
\begin{equation}\label{pw6-vct}
\begin{split}
\int_{\hat{T}}^t\abs{\left(\norm{v}_{L^6}^6\right)_t}dt\le C.
\end{split}\end{equation}
Combining \eqref{pw6-vcf} with \eqref{pw6-vct} leads to
$
\norm{v}_{L^6}\to 0
~~\text{as}~~t\to\infty,
$
which together with the interpolation inequality, \eqref{pw6-L^4} and \eqref{pw6-vin} implies
\begin{equation*}
\norm{v}_{L^p}\to 0
~~\text{as}~~t\to\infty,\ \ 2<p<\infty.
\end{equation*}
This along with Lemma \ref{lnn} gives $\eqref{pw6-lb}$ and hence completes the proof of Theorem $\ref{pw6-th}$.

\section{Proof of Theorem $\ref{u_+=0}$ }

In this section, we prove the nonlinear stability of the traveling wave solution of (\ref{hp})-(\ref{Boundary}) with discontinuous initial data having large oscillations. The main result is that the solution of (\ref{hp})-(\ref{Boundary}) approaches the traveling wave solution $(U,V)(x-st)$, properly translated by an amount $x_0$, i.e.,
\begin{equation*}
\sup\limits_{x\in\R}\abs{(u,v)(x,t)-(U,V)(x+x_0-st)}\to 0,~~\mathrm{as}~~ t\to+\infty,
\end{equation*}
where $x_0$ satisfies the following identity derived from the ``conservation of mass'' principle
\begin{eqnarray*}
\int_{-\infty}^{+\infty}
\begin{pmatrix}
u_0(x)-U(x)\\
v_0(x)-V(x)
\end{pmatrix}
dx=x_0
\begin{pmatrix}
u_+-u_-\\
v_+-v_-
\end{pmatrix}+\beta r_1(u_-,v_-),
\end{eqnarray*}
where $r_1(u_-,v_-)$ denotes the first right eigenvector of the Jacobian matrix of (\ref{hp}) with in the absence of viscous terms evaluated at $(u_-,v_-)$, see details in \cite{smoller}. The coefficient $\beta$ yields the diffusion wave in general. Both $\beta$ and $x_0$ will be uniquely determined by the initial data $(u_0, v_0)$. For the stability of small-amplitude
shock waves of conservation laws with diffusion wave, $i.e.\ \beta\neq0 $, we refer to \cite{Liu09, sz93} . In the present paper, we will neglect the diffusion wave by assuming $\beta=0$ and we consider the stability of large-amplitude waves with large discontinuous data. Then by the conservation laws (\ref{hp}), we obtain that 
\begin{eqnarray}\label{p}
\begin{aligned}
&\int_{-\infty}^{+\infty}\bigg(\begin{array}{lll}\begin{aligned}
u(x,t)-U(x+x_0-st)\\v(x,t)-V(x+x_0-st)
\end{aligned}
\end{array}\bigg)dx
=\int_{-\infty}^{+\infty}\bigg(\begin{array}{lll}\begin{aligned}
u_0(x)-U(x+x_0)\\v_0(x)-V(x+x_0)
\end{aligned}
\end{array}\bigg)dx\\
&=\int_{-\infty}^{+\infty}\bigg(\begin{array}{lll}\begin{aligned}
u_0(x)-U(x)\\v_0(x)-V(x)
\end{aligned}
\end{array}\bigg)dx
+ \int_{-\infty}^{+\infty}\bigg(\begin{array}{lll}\begin{aligned}
U(x)-U(x+x_0)\\V(x)-V(x+x_0)
\end{aligned}
\end{array}\bigg)dx\\
&=\int_{-\infty}^{+\infty}\bigg(\begin{array}{lll}\begin{aligned}
u_0(x)-U(x)\\v_0(x)-V(x)
\end{aligned}
\end{array}\bigg)dx-x_0\bigg(\begin{array}{lll}\begin{aligned}u_+-u_-\\v_+-v_-
\end{aligned}
\end{array}\bigg).
\end{aligned}
\end{eqnarray}
This together with $\beta=0$ implies the zero integral of the initial perturbation
\begin{equation}\label{initial condition}
\D \int_{-\infty}^{+\infty}\bigg(\begin{array}{lll}
u_0(x)-U(x+x_0)
\\v_0(x)-V(x+x_0)
\end{array}\bigg) dx=
\bigg(\begin{array}{lll}
0\\
0
\end{array}\bigg).
\end{equation}
Then we  employ the technique of taking anti-derivative to decompose the solution as
\begin{equation}\label{2-5}
(u,v)(x,t)=(U,V)(x+x_0-st)+(\phi_x,\psi_x)(x,t).
\end{equation}
 That is
\begin{eqnarray*}
\begin{aligned}
(\phi(x,t), \psi(x,t)) =\int_{-\infty}^x (u(y,t)-U(y+x_0-st),v(y,t)-V(y+x_0-st))dy\\
\end{aligned}
\end{eqnarray*}
for $(x,t)\in \mathbb{R}\times \mathbb{R}_+ $. It then follows from (\ref{p}) that
\begin{equation*}
\phi(\pm\infty,t)=\psi(\pm\infty,t)=0, \ \ \mathrm{for \ all} \ \ t>0.
\end{equation*}
The initial perturbation $(\phi_0, \psi_0)(x)=(\phi(x,0), \psi(x,0))$ is thus given by
\begin{equation*}
(\phi_0, \psi_0)(x)=-\int_x^{\infty} (u_0(y)-U(y+x_0),
v_0(y)-V(y+x_0))dy,
\end{equation*}
which satisfies $(\phi_0, \psi_0)(\pm\infty)=0$ by the assumption \eqref{initial condition}.

Substituting $\eqref{2-5}$ into (\ref{hp}), using  $\eqref{traveling
wave equation}$ and integrating the system with respect to $x$, we
obtain that ~$(\phi,\psi)(x,t)$~satisfies
\begin{equation}\label{nonlinear system0}
\begin{cases}
\phi_t=D\phi_{xx}+\chi V\phi_x+\chi U\psi_x+\chi \phi_x\psi_x,~~t>0, ~~x\in\R,\\
\psi_t=\phi_x,
\end{cases}
\end{equation}
with initial perturbation
\begin{equation*}
(\phi_0, \psi_0)(x)=-\int_x^{\infty} (u_0(y)-U(y+x_0),
v_0(y)-V(y+x_0))dy,
\end{equation*}
and
\begin{equation}\label{ini0}
(\phi_0(x), \psi_{0}(x) )\in H^1(\mathbb{R}), \ \ \    \psi_{0x}(x)\in L^\infty(\mathbb{R}).
\end{equation}

We denote
\begin{equation*}
\begin{split}
X[0,T]:=\{&(\phi(x,t),\psi(x,t))\big|\phi\in L^\infty([0,T]; H^1), \phi_x\in L^\infty([0,T]; L^2)\cap L^2([0,T];H^1),    \\&\psi\in L^\infty([0,T]; H^1),\psi_x \in L^\infty([0,T]; L^2\cap L^\infty)\cap L^2([0,T];L^2)
\end{split}\end{equation*}
and
\begin{equation}\label{c0}
C_0:=\|\psi_0\|_1^2+\|\phi_0\|_{1}^2.
\end{equation}

\begin{definition} We say that $(\phi, \psi)$ is a weak solution of \eqref{nonlinear system0}-\eqref{ini0}, if $(\phi,\psi)\in X[0,\infty)$, and for all test functions $\Psi\in C_0^\infty(\mathbb{R}\times [0, \infty))$ satisfy that
\begin{equation}\label{de-weak1}
\begin{split}
\int_{\mathbb{R}}\phi_0\Psi_0(x)dx+\int_0^\infty\int_{\mathbb{R}}\left(\phi\Psi_t+D\phi_x\Psi_x\right)dxdt=\chi\int_0^\infty\int_{\mathbb{R}}\left( V\phi_x+ U\psi_x+\phi_x\psi_x\right)\Psi dxdt
\end{split}
\end{equation}
and
\begin{equation}\label{de-weak2}
\begin{split}
\int_{\mathbb{R}}\psi_0\Psi_0(x)dx+\int_0^\infty\int_{\mathbb{R}}\left(\psi\Psi_t-\phi_x\Psi\right)dxdt=0.
\end{split}
\end{equation}
\end{definition}
For the problem \eqref{nonlinear system0}-\eqref{ini0}, we have
the following results.
\begin{theorem}\label{global existence}
Let $u_+>0$ and the initial data satisfy \eqref{ini0}. There exists a  constant $\varepsilon>0$, such that if $C_0\leq\varepsilon$,
then the  problem \eqref{nonlinear system0}-\eqref{ini0} has a global weak solution in the sense of \eqref{de-weak1}-\eqref{de-weak2} satisfying
\begin{equation}\label{le-last}
\begin{cases}
\begin{split}
&\|\phi\|_{1}^2+\|\psi\|_1^2+\int_0^T\left(\|\phi_x(t)\|_{1}^2+\|\psi_x(t)\|^2\right)dt
\leq CC_0,\\[2mm]
&\sigma\int_{\R}({\phi^2_{t}}+\phi^2_{xx}) dx+\int_0^T\int_{\R}\sigma\phi^2_{xt}dxdt\le CC_0,\\[2mm]
&\sup_{t\in[0,\infty)}\|\psi(\cdot,t)_x\|_{L^\infty}\le C,
\end{split}\end{cases}
\end{equation}
where $\sigma=\sigma(t)=\min\{1,t\}$. Moreover, it follows that
\begin{equation}\label{long-time behavior}
\begin{split}
&\sup\limits_{x\in\mathbb{R}}\abs{\phi_x(x,t)}\to
0~~as~~t\to\infty,\\
&\sup\limits_{x\in\mathbb{R}} \norm{\psi_x(x,t)}_{L^p} \to
0~{\rm for~all}~2\le p<\infty~~as~~t\to\infty.
\end{split}
\end{equation}
\end{theorem}

In view of \eqref{2-5}, Theorem \ref{u_+=0} is a consequence of Theorem \ref{global existence}. Hence next we are devoted to proving Theorem \ref{global existence}. Similarly as before, we first mollify the (coarse) initial data $(\phi_0, \psi_0)$ as follows:
\begin{equation*}
\begin{split}
\phi^\delta_0=j^\delta*\phi_0, \ \ \ \psi_0^\delta=j^\delta*\psi_0,
\end{split}
\end{equation*}
where $j^\delta$ is the standard mollifying kernel of width $\delta$ (e.g. see \cite{Adams}).
Then we consider the following augmented system
\begin{equation}\label{nonlinear system}
\begin{cases}
\phi^\delta_t=D\phi^\delta_{xx}+\chi V\phi^\delta_x+\chi U\psi^\delta_x+\chi \phi^\delta_x\psi^\delta_x, \  t>0, ~~x\in\R,\\[1mm]
\psi^\delta_t=\phi^\delta_x,
\end{cases}
\end{equation}
with smooth initial perturbation functions $(\phi^\delta_0, \psi_0^\delta)$ which satisfies
\begin{equation}\label{1ini}
(\psi^\delta_0(x), \phi^\delta_0(x))\in H^3(\mathbb{R}),
\end{equation}
and
\begin{equation}\label{pw8mde0}
\|\phi^\delta_0\|_{1}^2+\|\psi^\delta_0\|_1^2\le \|\phi_0\|_{1}^2+\|\psi_0\|_1^2=C_0,
\end{equation}
where we have used \eqref{c0} and the following properties:
\begin{equation*}
\begin{split}
\norm{\partial_k \phi^\delta_0}_{w}\le \norm{\partial_k\phi_0}_{w},\ \norm{\partial_k\psi^\delta_0}_{w}\le \norm{\partial_k\psi_0}_{w} \ \text{for every} \ k=0,1,\  \delta>0.
\end{split}
\end{equation*}

Next, by standard approaches, we prove the local existence of solutions to the system \eqref{nonlinear system} with initial data $(\phi^\delta_0,\  \psi_0^\delta)$ satisfying (\ref{1ini}). Then, we shall show in a sequence of lemmas that these approximate solutions
satisfy some global {\it a priori} estimates, independently of the mollifying parameter $\delta$. By the continuation argument, we can get the global existence of $(\phi^\delta, \psi^\delta)$. Finally, we show that the limit of $(\phi^\delta, \psi^\delta)$ as $\delta \to 0$ is a global weak solution of the Cauchy problem \eqref{nonlinear system}-\eqref{1ini}, and thus Theorem \ref{global existence} is proved.

\subsection{A priori estimates for \eqref{nonlinear system}}
For simplicity, in this subsection, we still use $(\phi,  \psi)$ to represent the approximate solution $(\phi^\delta, \psi^\delta)$ and employ the technique of {\it a priori} assumption to derive the {\it a priori} estimates for the smooth solutions of \eqref{nonlinear system}-\eqref{1ini}.
To this end, we first assume that the solution $(\phi,\psi)$ satisfies for any $t\in[0, T]$ that
\begin{equation}\label{assup}
\|\phi\|_{1}^2+\|\psi\|_1^2\leq 2\kappa_0,
\end{equation}
where $\kappa_0$ is a positive constant. Then we derive the {\it a priori} estimates to obtain global solutions. Finally, we show the obtained global solutions in turn satisfy the above {\it a priori} assumption and close our argument.

We first give the $L^2$-estimate of $(\phi, \psi)$.

\begin{lemma}\label{Basic L^2 estimate}  With the conditions of Theorem \ref{global existence}, we let $(\phi, \psi)$ be a smooth solution of \eqref{nonlinear system} satisfying \eqref{assup}. Then there exists a positive constant $C$ independent of $t$  and $\delta$, such that
\begin{equation}\label{L^2 estimate}
\norm{\phi}^2+\norm{\psi}^2+\int_0^T\norm{\phi_x}^2dt\leq
C C_0+C\kappa_0\int_0^T\int_{\R}\psi_x^2dxdt.
\end{equation}
\end{lemma}
\begin{proof}
\par Multiplying the first equation of $\eqref{nonlinear system}$ by $\phi/U$ and the second  by $\chi\psi$ and adding these equalities, we obtain
\begin{equation*}
\frac{1}{2}\left(\frac{\phi^2}{U}\right)_t-\frac{\phi^2}{2}\left(\frac{1}{U}\right)_t+\left(\frac{\chi\psi^2}{2}\right)_t
=\frac{D\phi\phi_{xx}}{U}+\chi\left(\phi\psi\right)_x+ \frac{\chi
V\phi\phi_x}{U}+\frac{\chi\phi\phi_x\psi_x}{U}.
\end{equation*}
Noting that
\begin{equation*}
\frac{\phi^2}{2}\left(\frac{1}{U}\right)_t=-\frac{s\phi^2}{2}\left(\frac{1}{U}\right)_x,
\end{equation*}
\begin{equation*}
\frac{\phi\phi_{xx}}{U}=\left(\frac{\phi\phi_x}{U}\right)_x-\frac{\phi_x^2}{U}-\phi\phi_x\left(\frac{1}{U}\right)_x
=\left(\frac{\phi\phi_x}{U}\right)_x-\frac{\phi_x^2}{U}-\left(\frac{\phi^2}{2}\left(\frac{1}{U}\right)_x\right)_x+\frac{\phi^2}{2}\left(\frac{1}{U}\right)_{xx},
\end{equation*}
\begin{equation*}
\begin{split}
\frac{V\phi\phi_x}{U}=\frac{1}{2}\left(\frac{V\phi^2}{U}\right)_x-\frac{\phi^2}{2}\left(\frac{V}{U}\right)_x,
\end{split}\end{equation*}
we get
\begin{equation}\label{2.15}
\begin{split}
\frac{1}{2}\left(\frac{\phi^2}{U}+\chi\psi^2\right)_t+\frac{D\phi_x^2}{U}
=&\left(\chi\phi\psi+\frac{D\phi\phi_x}{U}+\frac{D U_x\phi^2}{2U^2}
+\frac{\chi V\phi^2}{2U}\right)_x\\[3mm]
&+\frac{\phi^2}{2}\left[\left(\frac{D}{U}\right)_{xx}-\left(\frac{s+\chi
V}{U}\right)_x\right]+\frac{\chi\phi\phi_x\psi_x}{U}.
\end{split}
\end{equation}
By using $\eqref{traveling wave equation}$ and the fact that $U_x<0$
and $0<u_+\leq U\leq u_-$, it can be checked that
\begin{equation}\label{zero}
\left(\frac{D}{U}\right)_{xx}-\left(\frac{s+\chi
V}{U}\right)_x=\frac{2u_+}{U^3}(s+\chi v_+)\cdot U_x<0.
\end{equation}
Substituting \eqref{zero} into \eqref{2.15} and integrating the
equation over $\R\times[0, T]$, we derive
\begin{equation*}
\begin{split}
&\frac{1}{2}\int_{\R}\left(\frac{\phi^2}{U}+\chi\psi^2\right)dx+D\int_0^T\int_{\R}\frac{\phi_x^2}{U}dxdt\\[3mm]
\le&\frac{1}{2}\int_{\R}\left(\frac{\phi_0^2}{U}+\chi\psi_0^2\right)dx+\chi\int_0^T\int_{\R}\frac{\phi_x\psi_x\phi}{U}dxdt \\[3mm]
\leq&\frac{\chi}{2}\|\psi_0\|^2+C\|\phi_0\|^2+\frac{D}{2}\int_0^T\int_{\R}
\frac{\phi_x^2}{U}dxdt+C\kappa_0\int_0^T\int_{\R}
\frac{\psi_x^2}{U}dxdt,
\end{split}
\end{equation*}
where we have used the Sobolev
inequality $\norm{f}_{L^\infty}^2\le 2\norm{f}\norm{f_x}$ and $\eqref{assup}$. Then, using $0<u_+\leq U\leq u_-$ and \eqref{pw8mde0}, we obtain
\begin{equation*}
\int_{\R}\left(\phi^2+\chi\psi^2\right)dx+D\int_0^T\int_{\R}\phi_x^2dxdt
\leq C(\|\psi_0\|^2+\|\phi_0\|^2)+C\kappa_0\int_0^T\int_{\R}\psi_x^2dxdt,
\end{equation*}
which implies \eqref{L^2 estimate} and the proof of Lemma \ref{Basic L^2 estimate} is completed.
\end{proof}

The next lemma gives the estimate of the first order derivatives of
$(\phi,\psi)$.
\begin{lemma}\label{H^1 estimate} Let $(\phi, \psi)$ be a smooth solution of \eqref{nonlinear system} satisfying \eqref{assup} under the conditions of Theorem \ref{global existence}. Then there exists a positive constant $C$ independent of $t$  and $\delta$, such that
\begin{equation}\label{eqn-2.25}
\|\phi\|^2_{1} +\|\psi\|_1^2+\int_0^T\left(\|\phi_x\|^2_{1}
+\|\psi_x\|^2\right)dt \leq  CC_0.
\end{equation}
\end{lemma}
\begin{proof}
\par Multiplying the first equation of $\eqref{nonlinear system}$ by $-\phi_{xx}/U$ and the second  by $-\chi\psi_{xx}$ and adding these equalities, we obtain
\begin{equation*}
-\frac{\phi_t\phi_{xx}}{U}-\chi\psi_t\psi_{xx}
=-\frac{D\phi_{xx}^2}{U}-\chi\left(\phi_x\psi_x\right)_x-\frac{\chi
V\phi_x\phi_{xx}}{U}-\frac{\chi\phi_x\psi_x\phi_{xx}}{U}.
\end{equation*}
Simple calculations give us that
\begin{equation*}
\begin{split}
-\frac{\phi_t\phi_{xx}}{U}=&-\left(\frac{\phi_t\phi_{x}}{U}\right)_x+\left(\frac{\phi_t}{U}\right)_x\phi_x\\
=&-\left(\frac{\phi_t\phi_{x}}{U}\right)_x+\frac{\phi_{xt}\phi_x}{U}+\left(\frac{1}{U}\right)_x\phi_t\phi_x\\
=&-\left(\frac{\phi_t\phi_{x}}{U}\right)_x+\left(\frac{\phi_x^2}{2U}\right)_t+\left(\frac{1}{U}\right)_x\frac{s\phi_x^2}{2}+\left(\frac{1}{U}\right)_x\phi_t\phi_x,\\[3mm]
\left(\frac{1}{U}\right)_x\phi_t\phi_x=&\left(\frac{1}{U}\right)_x\phi_x\left(D\phi_{xx}+\chi
V\phi_x+\chi
U\psi_x+\chi\phi_x\psi_x\right)\\
=&\left(\frac{D\phi_x^2}{2}\left(\frac{1}{U}\right)_x\right)_x-\frac{D\phi_x^2}{2}\left(\frac{1}{U}\right)_{xx}+\chi
V\left(\frac{1}{U}\right)_{x}\phi_x^2\\
&+\chi
U\left(\frac{1}{U}\right)_{x}\psi_x\phi_x+\chi\left(\frac{1}{U}\right)_{x}\phi_x^2\psi_x,\\
-\psi_t\psi_{xx}=&-\left(\psi_t\psi_{x}\right)_x+\left(\frac{\psi_x^2}{2}\right)_t,\\
-\frac{V\phi_x\phi_{xx}}{U}=&-\frac{1}{2}\left(\frac{V\phi_x^2}{U}\right)_x+\frac{\phi_x^2}{2}\left(\frac{V}{U}\right)_x.
\end{split}\end{equation*}
Thus we get from above inequalities that
\begin{equation}\label{2.151}
\begin{split}
\frac{1}{2}\left(\frac{\phi_x^2}{U}+\chi\psi_x^2\right)_t+\frac{D\phi_{xx}^2}{U}
=&\left(\frac{\phi_t\phi_{x}}{U}+\chi\psi_t\psi_x+\frac{DU_x\phi_x^2}{2U^2}-\chi\phi_x\psi_x-\frac{\chi
V\phi_x^2}{2U} \right)_x\\[3mm]
&+\frac{\phi_x^2}{2}\left[\left(\frac{D}{U}\right)_{xx}-\left(\frac{s+\chi
V}{U}\right)_x\right]+\frac{\chi V_x\phi_x^2}{U}\\[3mm]
 &-\chi
U\left(\frac{1}{U}\right)_{x}\psi_x\phi_x-\chi\left(\frac{1}{U}\right)_{x}\phi_x^2\psi_x-\frac{\chi\phi_x\psi_x\phi_{xx}}{U}.
\end{split}
\end{equation}
Integrating \eqref{2.151} over $\R\times[0, T]$ and using
\eqref{zero}, we obtain
\begin{equation*}
\begin{split}
&\frac{1}{2}\int_{\R}\left(\frac{\phi_x^2}{U}+\chi\psi_x^2\right)dx+D\int_0^T\int_{\R}\frac{\phi_{xx}^2}{U}dxdt\\[3mm]
\le&\frac{1}{2}\int\left(\frac{\phi_{0x}^2}{U}+\chi\psi_{0x}^2\right)dx
+\chi\int_0^T\int_{\R}\frac{
V_x\phi_x^2}{U}dxdt+\chi\int_0^T\int_{\R}\frac{U_x\psi_x\phi_x}{U}dxdt\\[3mm]
&+\chi\int_0^T\int_{\R}\frac{U_x\phi_x^2\psi_x}{U^2}dxdt-\chi\int_0^T\int_{\R}\frac{\phi_{xx}\phi_x\psi_x}{U}dxdt.
\end{split}\end{equation*}
Using the Cauchy-Schwarz inequality, we have
\begin{eqnarray*}
&&\int_{\R}\left(\frac{\phi_x^2}{U}+\chi\psi_x^2\right)dx+2D\int_0^T\int_{\R}\frac{\phi_{xx}^2}{U}dxdt\\[3mm]
&&\le\int_{\R}\left(\frac{\phi_{0x}^2}{U}+\chi\psi_{0x}^2\right)dx+C\int_0^T\int_{\R}\phi_x^2dxdt+C\int_0^T\int_{\R}
\psi_x^2dxdt\\[3mm]
&&\ \
+\frac{D}{2}\int_0^T\int_{\R}\frac{\phi_{xx}^2}{U}dxdt+C\int_0^T\int_{\R}\phi_x^2\psi_x^2  dxdt,
\end{eqnarray*}
where we have used the fact
$0<u_+\leq U\leq u_-,~\abs{U_x}\leq C,~\abs{V_x}\leq C$ due to Proposition \ref{etw}.
For the last term on the right-hand side of the above inequality, by the Sobolev
inequality $\norm{f}_{L^\infty}^2\le 2\norm{f}\norm{f_x}$ and \eqref{assup}, we have
\begin{equation}\label{doub}
\begin{split}
\int_0^T\int_{\R}\phi_x^2\psi_x^2  dxdt \le& \int_0^T \norm{\phi_x}^2_{L^\infty}\norm{\psi_x}^2  dt \le C\kappa_0 \int_0^T \norm{\phi_x}\norm{\phi_{xx}}  dt  \\[2mm]
\le& C\kappa_0 \int_0^T \left(\norm{\phi_x}^2+\norm{\phi_{xx}}^2\right)dt.
\end{split}
\end{equation}
The above two inequalities and the fact
$0<u_+\leq U\leq u_-$ yield that
\begin{equation}\label{eqn-2.16}
\begin{split}
&\int_{\R}\left(\phi_x^2+\psi_x^2\right)dx+\int_0^T\int_{\R}\phi_{xx}^2dxdt\\[2mm]
\leq &CC_0+C\int_0^T\int_{\R}{\psi_x^2}dxdt+C\kappa_0 \int_0^T \left(\norm{\phi_x}^2+\norm{\phi_{xx}}^2\right)dt\\[2mm]
\leq& CC_0 +C\int_0^T\int_{\R}{\psi_x^2}dxdt+C\kappa_0\int_0^T\int_{\R}\psi_x^2dxdt+C\kappa_0\int_0^T \norm{\phi_{xx}}^2 dt\\[2mm]
\leq& CC_0 +C\int_0^T\int_{\R}{\psi_x^2}dxdt+C\kappa_0\int_0^T \norm{\phi_{xx}}^2  dt,
\end{split}\end{equation}
where \eqref{L^2 estimate} has been used.

Next, we claim
\begin{equation}\label{eqn-2.17}
\int_0^T\int_{\R} \psi_x^2dxdt\leq C \left(C_0+\kappa_0\int_0^T \norm{\phi_{xx}}^2  dt\right).
\end{equation}
Indeed multiplying the first equation of $\eqref{nonlinear system}$ by
$\psi_x$, we get
\begin{equation}\label{eqn-2.18}
\chi U\psi_x^2=\phi_{t}\psi_x-D\phi_{xx}\psi_x-\chi
V\phi_x\psi_x-\chi\phi_x\psi_x^2.
\end{equation}
Integrating \eqref{eqn-2.18} over  $\R\times[0, T]$,
using the fact $\psi_{xt}=\phi_{xx}$ and following results
\begin{equation*}
\begin{split}
\phi_{t}\psi_x&=(\phi\psi_x)_t-\phi\psi_{xt}=(\phi\psi_x)_t-\phi\phi_{xx}=(\phi\psi_x)_t-(\phi\phi_x)_x+\phi_x^2,\\[2mm]
\phi_{xx}\psi_x&=\psi_{xt}\psi_x=\frac{1}{2}(\psi_x^2)_t,
\end{split}\end{equation*}
we obtain
\begin{equation*}
\begin{split}
&\frac{D}{2}\int_{\R}\psi_x^2dx+\chi\int_0^T\int_{\R} U\psi_x^2dxdt \\[2mm]
&=\frac{D}{2}\int_0^\infty\psi_{0x}^2dx+\int_{\R}\phi\psi_xdx-\int_{\R}\phi_0\psi_{0x}dx\\[2mm]
&\ \ \ +
\int_0^T\int_{\R}\phi_x^2dxdt-\chi\int_0^T\int_{\R}
V\phi_x\psi_xdxdt-\chi\int_0^T\int_{\R}\phi_x\psi_x^2dxdt\\[2mm]
&\leq
\frac{D+1}{2}\int_{\R}\psi_{0x}^2dx+\frac{1}{2}\int\phi_0^2dx
+\frac{1}{D}\int_{\R}\phi^2dx+\frac{D}{4}\int_{\R}\psi^2_xdx\\[2mm]
&\quad+C\int_0^T\int_{\R}
\phi_x^2dxdt+\frac{\chi}{2}\int_0^T\int_{\R}
U\psi_x^2dxdt+C\int_0^T\int_{\R}\phi_x^2\psi_x^2  dxdt ,
\end{split}
\end{equation*}
where we have used the Cauchy-Schwarz inequality and the fact
$0<u_+\leq U\leq u_-,~\abs{V}\leq C$. From
this inequality and  the fact $0<u_+\leq U\leq u_-$, \eqref{L^2 estimate} and \eqref{doub}, we have that
\begin{equation*}
\begin{split}
\int_{\R}\psi_x^2dx+\int_0^T\int_{\R}
\psi_x^2dxdt
\leq& C\left(\int_{\R}\psi_{0x}^2dx
+\int_{\R}\phi_{0x}^2dx\right)+C\int_{\R}\phi^2dx\\
&+C\int_0^T\int_{\R}\phi_x^2dxdt+C\kappa_0\int_0^T \norm{\phi_{xx}}^2 dt +C\kappa_0\int_0^T \norm{\phi_x}^2dt\\[2mm]
\leq&
CC_0+C\kappa_0\int_0^T\int_{\R}\psi_x^2dxdt+C\kappa_0\int_0^T \norm{\phi_{xx}}^2  dt.
\end{split}
\end{equation*}
Setting $\kappa_0$ suitably small such that $C\kappa_0\le\frac12$,
we get \eqref{eqn-2.17}. Then substituting \eqref{eqn-2.17} into
\eqref{eqn-2.16} and choosing $\kappa_0$ suitably small, we have
\begin{equation*}
\begin{split}
\int_{\R}\left(\phi_x^2+\psi_x^2\right)dx+\int_0^T\int_{\R}\phi_{xx}^2dxdt
 \leq CC_0,
\end{split}\end{equation*}
which together with \eqref{L^2 estimate} and \eqref{eqn-2.17} gives \eqref{eqn-2.25}.
\end{proof}

Now, taking $C_0$ sufficiently small such that $CC_0\le \kappa_0$, we immediately get from \eqref{eqn-2.25} that
\begin{equation*}
\|\phi\|^2_{1} +\|\psi\|_1^2+\int_0^T\left(\|\phi_x\|^2_{1}
+\|\psi_x\|^2\right)dt \leq \kappa_0,
\end{equation*}
which closes the a priori assumption \eqref{assup}.

Next, we derive the appropriate estimates for the second order derivative of $\phi$.


\begin{lemma}\label{sec}Let the conditions of Theorem \ref{global existence} hold and $(\phi, \psi)$ be a smooth solution of \eqref{nonlinear system} satisfying \eqref{assup}.
Then there exists a positive constant $C$ independent of $t$  and $\delta$, such that
\begin{equation}\label{eqn-2.31}
\sigma\int_{\R}({\phi^2_{t}}+\phi^2_{xx}) dx+D\int_0^T\int_{\R}\sigma\phi^2_{xt}dxdt\le CC_0,
\end{equation}
where $\sigma=\sigma(t)=\min\{1,t\}$ and $C$ is positive constant independent of $t$.
\end{lemma}
\begin{proof}
\par We differentiate the first equation of \eqref{nonlinear system} with respect to $t$ to get
\begin{equation*}
\phi_{tt}=D\phi_{xxt}-\chi sV_x\phi_x+\chi V\phi_{xt}-\chi s U_x\psi_x+\chi U\psi_{xt}+\chi\phi_{xt}\psi_x+\chi\phi_x\psi_{xt}.
\end{equation*}
Multiplying the above equality by
$\sigma\phi_{t}$, one gets
\begin{equation}\label{2.152}
\begin{split}
{\sigma\phi_{tt}\phi_{t}}=&{D\sigma\phi_{xxt}\phi_{t}}+\chi\sigma U\psi_{xt}\phi_{t}\\
&+{\chi\sigma\phi_{t}}(
- sV_x\phi_x+ V\phi_{xt}-s U_x\psi_x+\phi_{xt}\psi_x+\phi_x\psi_{xt}).
\end{split}
\end{equation}
Integrating \eqref{2.152}  over $\R\times[0, T]$
and rearranging the resulting equation, we get
\begin{equation}\label{2.31}
\begin{split}
&\frac{1}{2}\int_{\R}{\sigma\phi^2_{t}} dx+D\int_0^T\int_{\R}\sigma\phi^2_{xt}dxdt
\\[2mm]
&=\frac{1}{2}\int_0^{\sigma(t)}\int_{\R}\phi^2_{t}dxdt+\chi\int_0^T\int_{\R}\sigma U\phi_{xx}\phi_{t}dxdt+\chi\int_0^T\int_{\R}\sigma\phi_{xt}\psi_x\phi_t dxdt
\\[2mm]
& \ \ \ +\chi\int_0^{t}\int_{\R}{\sigma\phi_{t}}(
- s U_x\psi_x-sV_x\phi_x+ V\phi_{xt})dxdt+\chi\int_0^T\int_{\R}\sigma\phi_{t}\phi_x\phi_{xx}dxdt\\[2mm]
&=I_1+I_2+I_3+I_4+I_5,
\end{split}\end{equation}where we have used the integration by parts and the fact $\psi_{xt}=\phi_{xx}$ due to the second equation of $\eqref{nonlinear system}$. Because $\left|V\right|$ and $|U|$ are all bounded, we
get by the first equation of \eqref{nonlinear system}, \eqref{doub} and \eqref{eqn-2.25} that
\begin{equation}\label{2.1520}
\begin{split}
\int_0^T\int_{\R}\phi^2_{t}dxdt\le& C\int_0^T\int_{\R} (\phi^2_{xx}+V^2 \phi^2_{x}+ U^2 \psi^2_{x})dxdt+C\int_0^T\int_{\R}\phi^2_{x}\psi^2_{x}dxdt\\[2mm]
\le& C\int_0^T\int_{\R} (\phi^2_{xx}+\phi^2_{x}+\psi^2_{x})dxdt+CC_0\int_0^T\int_{\R} (\phi^2_{xx}+\phi^2_{x})dxdt\\[2mm]
\le& CC_0.
\end{split}
\end{equation}
Then, $I_1$ can be bounded as
$
I_1\le \int_0^1\int_{\R}\phi^2_{t}dxdt
\le CC_0.
$
For $I_2$, by the Cauchy-Schwartz inequality, \eqref{eqn-2.25} and \eqref{2.1520}, we have
$
I_2\le C\int_0^T\int_{\R}\sigma (\phi^2_{xx}+\phi^2_{t})dxdt \le CC_0.
$
For $I_3$, using the Cauchy-Schwartz and Sobolev inequalities, \eqref{eqn-2.25} and \eqref{2.1520}, we have
\begin{equation*}
\begin{split}
I_3\le& \frac{D}{8}\int_0^T\int_{\R}\sigma\phi^2_{xt}dxdt+C\int_0^T\sigma\norm{\phi_t}^2_{L^\infty}\norm{\psi_x}^2dt\\[2mm]
\le& \frac{D}{8}\int_0^T\int_{\R}\sigma\phi^2_{xt}dxdt+CC_0\int_0^T\sigma\norm{\phi_t}\norm{\phi_{xt}}dt\\[2mm]
\le& \frac{D}{4}\int_0^T\int_{\R}\sigma\phi^2_{xt}dxdt+CC_0\int_0^T\sigma\norm{\phi_t}^2dt\\[2mm]
\le& \frac{D}{4}\int_0^T\int_{\R}\sigma\phi^2_{xt}dxdt+CC_0.
\end{split}
\end{equation*}
Since $\left|V\right|$, $|U_x|$ and $|V_x|$ are all bounded, we
get by the Cauchy-Schwartz inequality and \eqref{eqn-2.25} that
\begin{equation*}
\begin{split}
I_4\le&
\frac{D}{4}\int_0^T\int_{\R}\sigma\phi^2_{xt}dxdt+C\int_0^T\int_{\R} \sigma(\phi^2_{t}+\phi^2_{x}+ \psi^2_{x})dxdt\\[2mm]
\le& \frac{D}{4}\int_0^T\int_{\R}\sigma\phi^2_{xt}dxdt+CC_0.
\end{split}
\end{equation*}
Using the integration by parts, the Cauchy-Schwartz and Sobolev inequalities and \eqref{eqn-2.25}, we have
\begin{equation*}
\begin{split}
I_5=&-\frac{\chi}{2}\int_0^T\int_{\R}\sigma\phi_{xt}\phi^2_xdxdt\\[2mm]
\le& \frac{D}{4}\int_0^T\int_{\R}\sigma\phi^2_{xt}dxdt+C\int_0^T\sigma\norm{\phi_x}^2_{L^\infty}\norm{\phi_x}^2dxdt\\[2mm]
\le& \frac{D}{4}\int_0^T\int_{\R}\sigma\phi^2_{xt}dxdt+CC_0\int_0^T\sigma\norm{\phi_x}^2_1dxdt\\[2mm]
\le& \frac{D}{4}\int_0^T\int_{\R}\sigma\phi^2_{xt}dxdt+CC_0.
\end{split}
\end{equation*}
Substituting the estimates of $I_1-I_5$ into \eqref{2.31},
one has
\begin{equation}\label{eqn-2.27}
\begin{split}
\int_{\R}{\sigma\phi^2_{t}} dx+D\int_0^T\int_{\R}\sigma\phi^2_{xt}dxdt\le CC_0,
\end{split}\end{equation}
which, combined with \eqref{nonlinear system}, the Cauchy-Schwartz and Sobolev inequalities and \eqref{eqn-2.25} gives
\begin{equation}\label{uxx}
\begin{split}
\sigma\int_{\R}\phi^2_{xx}dx\le& C\sigma\int_{\R} (\phi^2_{t}+V^2 \phi^2_{x}+ U^2 \psi^2_{x})dx+C\sigma\int_{\R}\phi^2_{x}\psi^2_{x}dx\\[2mm]
\le& C\sigma\int_{\R} (\phi^2_{t}+\phi^2_{x}+\psi^2_{x})dx+CC_0\sigma\int_{\R} (\phi^2_{xx}+\phi^2_{x})dx\\[2mm]
\le& CC_0+CC_0\sigma\int_{\R} \phi^2_{xx}dx.
\end{split}
\end{equation}
By choosing $C_0$ sufficiently small, we get from \eqref{uxx} that
$\sigma\int_{\R}\phi^2_{xx}dx\le CC_0,$
which together with \eqref{eqn-2.27} leads to \eqref{eqn-2.31}. Thus, the proof of Lemma \ref{sec} is completed.
\end{proof}

\begin{lemma}Let the conditions of Theorem \ref{global existence} hold and $(\phi, \psi)$ be a smooth solution of \eqref{nonlinear system}.
Then it holds that
\begin{equation*}
\begin{split}
\sup\limits_{x\in\mathbb{R}}\abs{u(x,t)-U}\to 0
~~as~~t\to\infty.
\end{split}
\end{equation*}
Furthermore there exists a positive constant $\hat{T}>1$, such that $
u\geq \frac{u_+}{2}>0 $ for any $t\geq\hat{T}$.
\end{lemma}
\begin{proof}
\par
From $\sigma=1$ for $t\geq1$, $\psi_{xt}=\phi_{xx}$ and
$\eqref{le-last}$, we have
\begin{equation*}
\int_1^\infty\left(\|\phi_x\|^2
+\|\phi_{xt}\|^2+\|\psi_x\|^2
+\|\psi_{xt}\|^2\right)dt\le C,
\end{equation*}
which implies that
\begin{equation}\label{pwlt}
\norm{\phi_x(\cdot,t),\psi_x(\cdot,t)}\to 0
~~\text{as}~~t\to\infty.
\end{equation}
Hence, for all $x\in\mathbb{R}$, $t>1$,
\begin{equation*}
\begin{split}
\phi_x^2(x,t)&=2\left|\int^{\infty}_x\phi_x\phi_{xx}(y,t)dy\right|\\
             &\leq 2\left(\int_{\R}\phi_x^2dy\right)^{1/2}\left(\int_{\R}\phi_{xx}^2dy\right)^{1/2}\\
              &=2\left(\int_{\R}\phi_x^2dy\right)^{1/2}\left(\int_{\R}\sigma\phi_{xx}^2dy\right)^{1/2}
             \leq C\norm{\phi_x(\cdot,t)}\to 0 ~~\mathrm{as}~~ t\to\infty,
\end{split}
\end{equation*}
where we have used $\sigma(t)=1$ for $t>1$, \eqref{eqn-2.31} and \eqref{pwlt}. Thus,
$
\sup\limits_{x\in\mathbb{R}}\abs{\phi_x(x,t)}\to
0~~as~~t\to\infty,
$
which together with \eqref{2-5} leads to
$
\sup\limits_{x\in\mathbb{R}}\abs{u(x,t)-U}\to 0
~~as~~t\to\infty.
$
This implies that there exists a positive constant $\hat{T}>1$, such that for any $t\geq\hat{T}$,
$
\sup\limits_{x\in[\hat{T}, \infty)}\abs{u(x,t)-U}\le \frac{u_+}{2},
$
which, along with $U>u_+>0$ gives
$
u\ge \frac{u_+}{2}>0,\ {\rm for } \ t\geq\hat{T}
$
and hence completes the proof.
\end{proof}
We now proceed to derive a uniform (in time) upper bound for $\psi_x$.
\begin{lemma}\label{secl}Assume  the conditions of Theorem \ref{global existence} hold.
Let $(\phi, \psi)$ be a smooth solution of \eqref{nonlinear system}.
Then there exists a positive constant $C$ independent of $t$  and $\delta$, such that
\begin{equation}\label{seclp}
\norm{\psi_x}_{L^\infty} \le C.
\end{equation}
\end{lemma}
\begin{proof}
It follows from \eqref{nonlinear system} that
\begin{equation*}
\begin{split}
\psi_{xt}=\phi_{xx}=\frac{1}{D}(\phi_{t}-\chi V\phi_x-\chi U\psi_x-\chi \phi_x\psi_x),
\end{split}\end{equation*}
which together with \eqref{2-5} gives
\begin{equation*}
\begin{split}
\psi_{xt}+\frac{\chi}{D} u\psi_x=&\frac{1}{D}(\phi_{t}-\chi V\phi_x)
\le C\norm{\phi_{t}}_{L^\infty}+C\norm{\phi_x}_{L^\infty}.
\end{split}\end{equation*}
Taking
\begin{equation}\label{le-def}
\begin{split}
\begin{cases}
y=\psi_x, \ \ \alpha(t)=\frac{\chi}{D} u, \ \ \beta=\frac{\chi u_+}{2D}, \ \ T_1=\hat{T}>1, \ p=2,\\
g(t)=C\norm{\phi_{t}}_{L^\infty}+C\norm{\phi_x}_{L^\infty}
\end{cases}
\end{split}\end{equation}
in Lemma \ref{le-ine}, we have
\begin{equation}\label{l-in}
\begin{split}
\psi_x\le \norm{\psi_{0x}}_{L^\infty}+\left(1+\frac{2D}{\chi u_+}\right)(\norm{g}_{L^1(0,\hat{T})}+\norm{g}_{L^2(\hat{T},T)}).
\end{split}\end{equation}
For $0\leq t \leq \hat{T}$, we have from the H\"older and Sobolev inequalities, \eqref{eqn-2.31} and \eqref{2.1520} that
\begin{equation*}
\begin{split}
&\int_0^{\hat{T}}\norm{\phi_{t}}_{L^\infty}dt\\
&\le \sqrt{2}\int_0^{\hat{T}}\norm{\phi_{t}}^\frac{1}{2}\norm{\phi_{xt}}^\frac{1}{2}dt\\
&\le \sqrt{2}\left(\int_0^{\hat{T}}\norm{\phi_{t}}^2dt\right)^\frac{1}{4}\left(\int_0^{\hat{T}}\sigma(t)\norm{\phi_{xt}}^2dt\right)^\frac{1}{4}
\left(\int_0^{\hat{T}}\sigma(t)^{-\frac{1}{2}}dt\right)^\frac{1}{2}\\
&\le C C_0^\frac12
\left(\int_0^1\sigma(t)^{-\frac{1}{2}}dt+\int^{\hat{T}}_1 \sigma(t)^{-\frac{1}{2}}dt\right)^\frac{1}{2}\\
&\le C C_0^\frac12\left(\int_0^{\hat{T}}t^{-\frac{1}{2}}dt+\hat{T}-1\right)^\frac{1}{2}\le C.
\end{split}\end{equation*}
Using the H\"older and Sobolev inequalities again, we have from \eqref{eqn-2.25} that
\begin{equation*}
\begin{split}
\int_0^{\hat{T}}\norm{\phi_{x}}_{L^\infty}dt\le& \sqrt{2}\int_0^{\hat{T}}\norm{\phi_{x}}^\frac{1}{2}\norm{\phi_{xx}}^\frac{1}{2}dt\\
\le& \sqrt{2}\left(\int_0^{\hat{T}}\norm{\phi_{x}}^2dt\right)^\frac{1}{4}\left(\int_0^{\hat{T}}\norm{\phi_{xx}}^2dt\right)^\frac{1}{4}
\left(\int_0^{\hat{T}}1 dt\right)^\frac{1}{2}
\le C.
\end{split}\end{equation*}

The above inequalities in combination with \eqref{le-def} gives
\begin{equation}\label{gl1}
\begin{split}
\frac{2D}{\chi u_+}\norm{g}_{L^1(0,\hat{T})}\le C.
\end{split}\end{equation}
For $\hat{T}\le t \le T$, one deduces from the H\"older and Sobolev inequalities, \eqref{eqn-2.31} and \eqref{2.1520} that
\begin{equation*}
\begin{split}
\int^T_{\hat{T}}\norm{\phi_{t}}_{L^\infty}^2dt\le& 2\int^T_{\hat{T}}\norm{\phi_{t}}\norm{\phi_{xt}}dt\\[2mm]
\le& 2\left(\int^T_{\hat{T}}\norm{\phi_{t}}^2dt\right)^\frac{1}{2}\left(\int^T_{\hat{T}}\norm{\phi_{xt}}^2dt\right)^\frac{1}{2}\\[2mm]
\le& 2\left(\int^T_{\hat{T}}\norm{\phi_{t}}^2dt\right)^\frac{1}{2}\left(\int^T_{\hat{T}}\sigma(t)\norm{\phi_{xt}}^2dt\right)^\frac{1}{2}
\le C,
\end{split}\end{equation*}
where we have used the fact that $\sigma(t)=1$ for  $\hat{T}\le t\le T$. In a similar way, we have
\begin{equation*}
\begin{split}
\int^T_{\hat{T}}\norm{\phi_{x}}_{L^\infty}^2dt\le& 2\int^T_{\hat{T}}\norm{\phi_{x}}\norm{\phi_{xx}}dt\\
\le& 2\left(\int^T_{\hat{T}}\norm{\phi_{x}}^2dt\right)^\frac{1}{2}\left(\int^T_{\hat{T}}\norm{\phi_{xx}}^2dt\right)^\frac{1}{2}
\le C.
\end{split}\end{equation*}
The above two inequalities yields that
$
\frac{2D}{\chi u_+}\norm{g}_{L^2(\hat{T}, T)}\le C.
$
This along with \eqref{gl1} and \eqref{l-in} completes the proof of Lemma \ref{secl}.
\end{proof}


\subsection{Proof of Theorem $\ref{global existence}$}
Now we turn to prove Theorem \ref{global existence}. It first follows from \eqref{eqn-2.25}, \eqref{eqn-2.31} and \eqref{seclp} that
\begin{equation}\label{1le-last}
\begin{cases}
\begin{split}
&\|\phi^\delta\|^2_{1} +\|\psi^\delta\|_1^2+\int_0^T\left(\|\phi^\delta_x\|^2_{1}
+\|\psi^\delta_x\|^2\right)dt \leq CC_0,\\
&\sigma(\|\phi^\delta_{t}\|^2+\|\phi^\delta_{xx}\|^2)+D\int_0^T\sigma\|\phi^\delta_{xt}\|^2dt\le CC_0,\\
&\norm{\psi^\delta_x}_{L^\infty}\le C,
\end{split}
\end{cases}
\end{equation}
which together with \eqref{2.1520} leads to
\begin{equation*}
\begin{cases}
&(\phi^\delta, \psi^\delta)\in L^\infty([0,\infty), H^1(\mathbb{R})),  (\phi^\delta_t, \psi^\delta_t)\in L^2([0,\infty), L^2(\mathbb{R})), \\[2mm]
&(\phi^\delta_x, \psi^\delta_x)\in L^2([0,\infty), L^2(\mathbb{R})), \ \ \  \phi^\delta_x\in L^\infty((0,\infty), L^2(\mathbb{R})),\\[2mm]
&\phi_{xx}^\delta\in  L^2([0,\infty), L^2(\mathbb{R}))\cap L^\infty((0,\infty),L^2(\mathbb{R})),\ \ \ \phi^\delta_{xt}\in L^2((0,\infty), L^2(\mathbb{R})).
\end{cases}
\end{equation*}
By the Aubin-Lions-Simon lemma, we can extract a subsequence, still denoted by $(\phi^\delta,\  \psi^\delta)$, such that the following convergence hold as $\delta\to 0$
\begin{equation*}
\begin{cases}
(\phi^\delta, \psi^\delta)(\cdot,t)\to (\phi, \psi)(\cdot,t)\ \text{strongly  in} \ C([0,\infty), C(\mathbb{R})), \\[2mm]
\phi_x^\delta(\cdot,t)\to \phi_x(\cdot,t)\ \text{strongly  in} \ C((0,\infty), C(\mathbb{R})),\\[2mm]
\phi^\delta_{xx}(\cdot,t)\to \phi_{xx}(\cdot,t)\ \text{weakly  in} \ L^2([0,\infty),L^2(\mathbb{R})),\\[2mm]
\psi^\delta_x(\cdot,t)\to \psi_x(\cdot,t)\ \text{weakly  in} \ L^2([0,\infty),L^2(\mathbb{R})).
\end{cases}
\end{equation*}
Thus, it is easy to show that the limit function $(\phi, \psi)$  is indeed a weak solution of the system \eqref{nonlinear system0}-\eqref{ini0} and inherits all the bounds of \eqref{1le-last}. Thus, \eqref{le-last} is proved.

To complete the proof of  Theorem $\ref{global existence}$,
we only need to prove $\eqref{long-time behavior}$. For all $2\leq p<\infty$, we have from \eqref{le-last} and \eqref{pwlt} that
\begin{equation*}
\norm{\psi_x(x,t)}_{L^p}\le \norm{\psi_x(x,t)}_{L^\infty}^{\frac{p-2}{p}}\norm{\psi_x(x,t)}^{\frac{2}{p}}_{L^2}
\leq C\norm{\psi_x(\cdot,t)}^{\frac{2}{p}}\to 0 ~~\mathrm{as}~~ t\to\infty.
\end{equation*}
Hence $\eqref{long-time behavior}$ is proved and the proof of Theorem $\ref{global existence}$ is completed.

\section{Numerical verifications and predictions}
In this section, we shall numerically verify our results and further exploit the impact of the regularity of initial data on the regularity of solutions. For brevity, we shall take the case considered in Theorem \ref{u_+=0} as a target for simulations only, but similar conclusions apply to the case considered in Theorem \ref{pw6-th}.  We have shown in Theorem \ref{u_+=0} that if the initial data satisfy \eqref{pw6-ict1}, then the system \eqref{hp} has a global weak  solution $(u,v)$ where $u$ converges to a shifted traveling wave profile uniformly as time tends to infinity and $v$ converges to a shifted traveling wave profile in $L^p$-norm for $2\leq p<\infty$. However, it is unknown whether or not the regularity of weak solutions shown Theorem $\ref{u_+=0}$ can be improved. {\color{black}It has been shown in the existing literature (cf. \cite{jin13, Li09, Wang-xiang-yu}) that the initial data has $H^s$-regularity ($s\geq 1$), then strong solutions in $H^s$ can be obtained}. In this paper,  we focus on the initial data with lower regularity in $L^p$-space which in particular allows the discontinuous data. In our simulations, we shall first numerically illustrate the convergence of traveling wave profiles to verify our results. Second, we shall check if the regularity of (weak) solutions can be improved further to some extend. We perform the simulations in the interval $[0, 400]$ with Dirichlet boundary condition to mimic the whole space $\R$. We assume the asymptotic states are  $u_-=2, u_+=1, v_-=\frac{3-\sqrt{3}}{2}, v^+=1$ satisfying the relation \eqref{R-H condition}  and set the initial value $(u_0, v_0)$ as
\begin{eqnarray}\label{ini}
u_0(x)=\begin{cases}2, & 0\leq x\leq 50\\1, & 50< x\leq 400  \end{cases}, \ \ v_0(x)=\begin{cases}\frac{3-\sqrt{3}}{2}, & 0\leq x\leq 50\\2, & 50< x\leq 400\end{cases}
\end{eqnarray}
as plotted in Fig.\ref{fig1} (see red color dashed line).
\begin{figure}[!htbp]
\centering
\includegraphics[width=7.5cm]{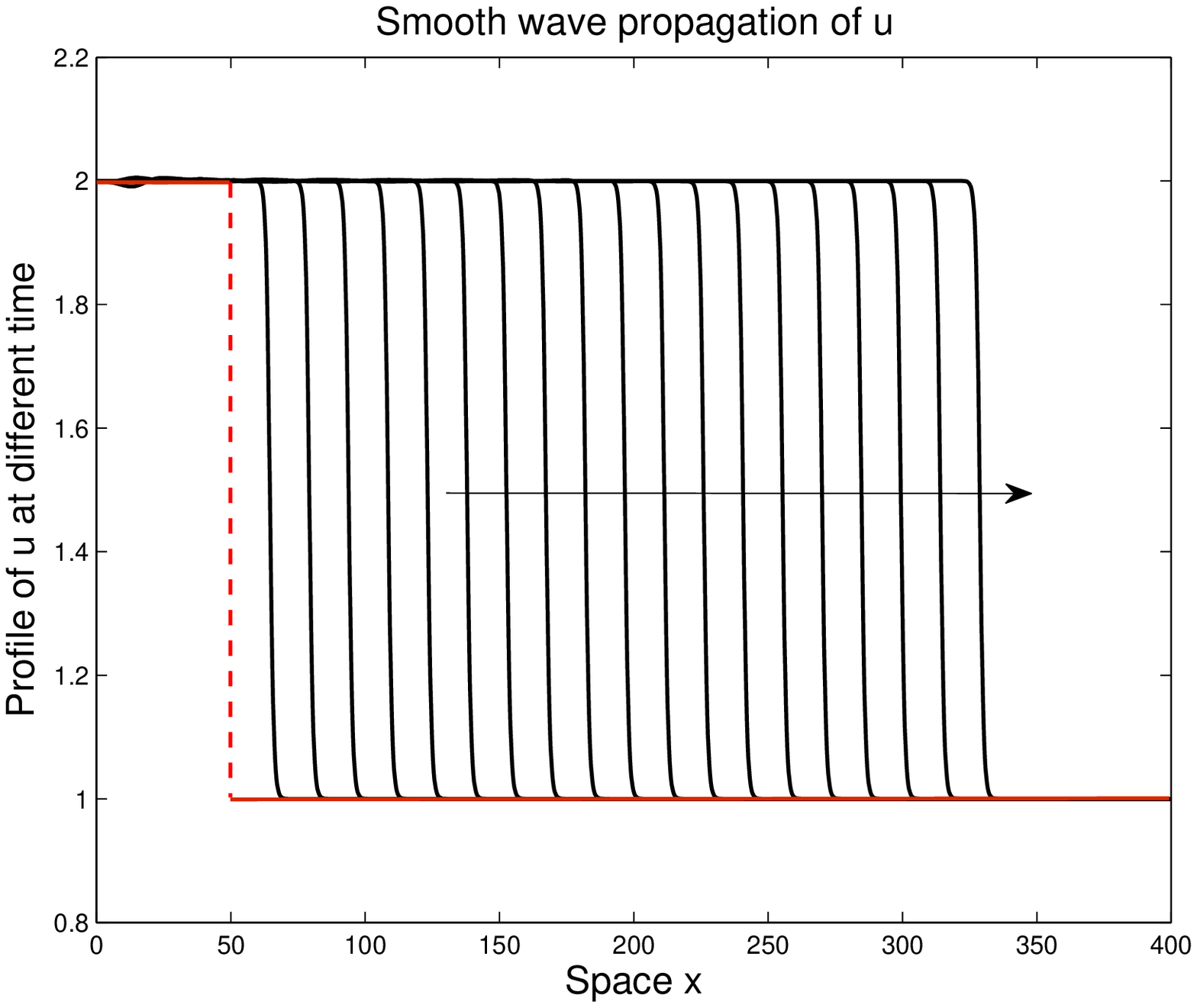}
\includegraphics[width=7.5cm]{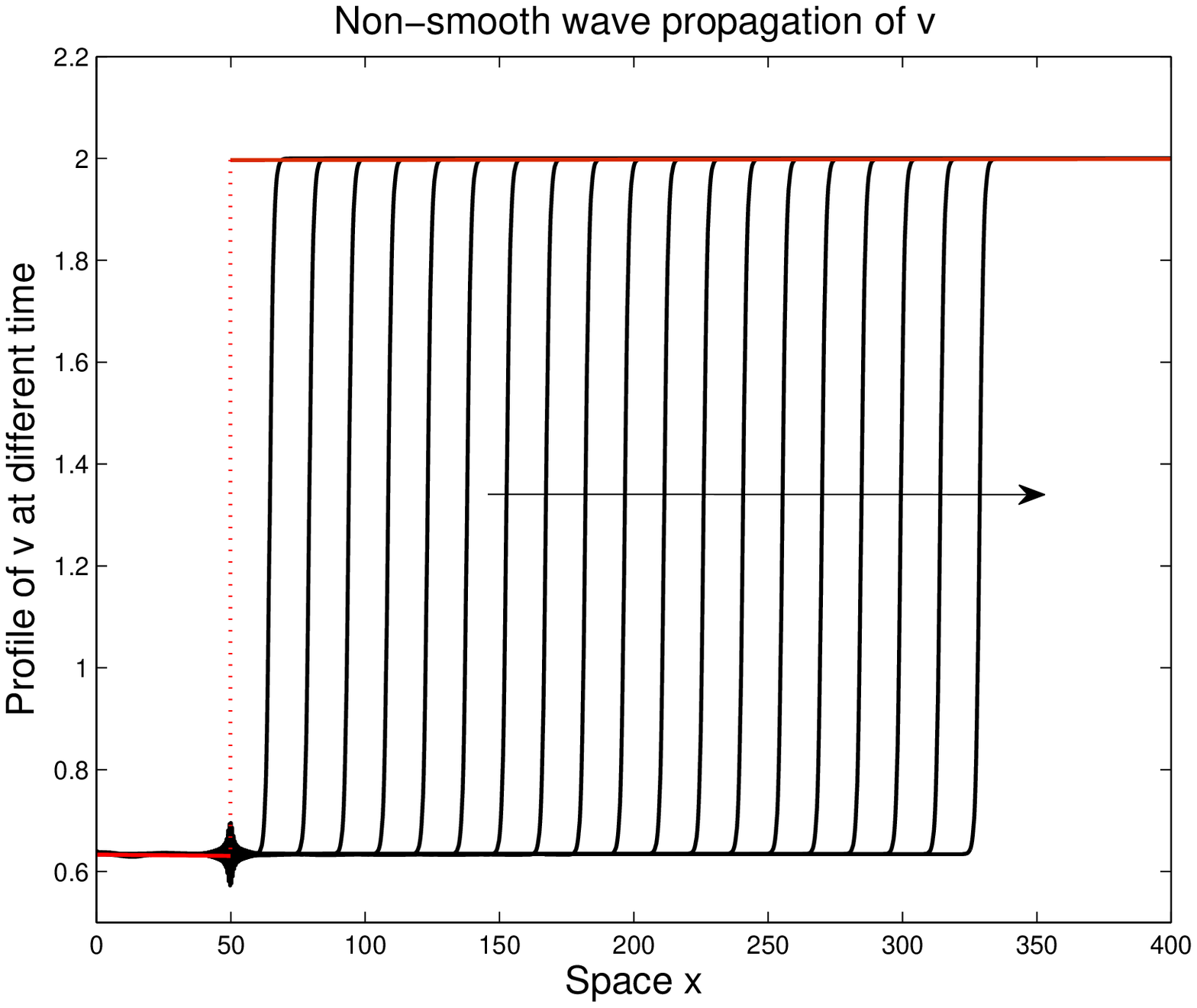}

\caption{Numerical illustration of stabilization of smooth wave profile formation for solution component $u$ and non-smooth wave profile formation for solution component $v$ to system \eqref{hp} with discontinuous initial data given by \eqref{ini} in an interval $[0,400]$, where we choose $\chi=1$. Each curve represents the solution (wave) profile at a certain time starting at $t= 0$ (red dashed curve) and spaced by $t=20$. The arrow indicates that the wave propagates from the left to the right. The amplified visualization of the fuzzy part of $v$-profile near the discontinuous point $x=50$ is plotted in Fig.\ref{fig2}. }
\label{fig1}
\end{figure}

\begin{figure}[!htbp]
\centering
\includegraphics[width=7cm]{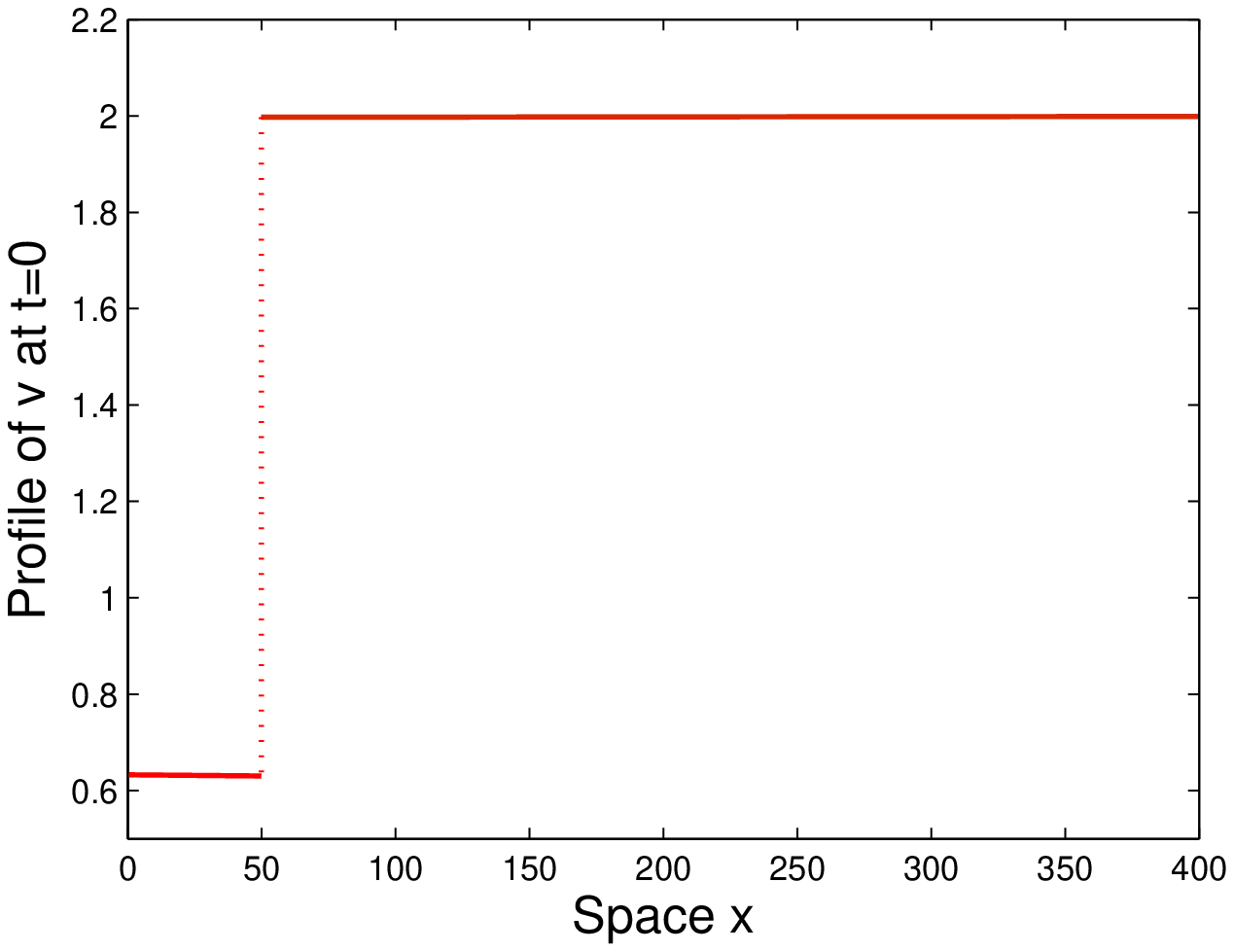}
\includegraphics[width=7cm]{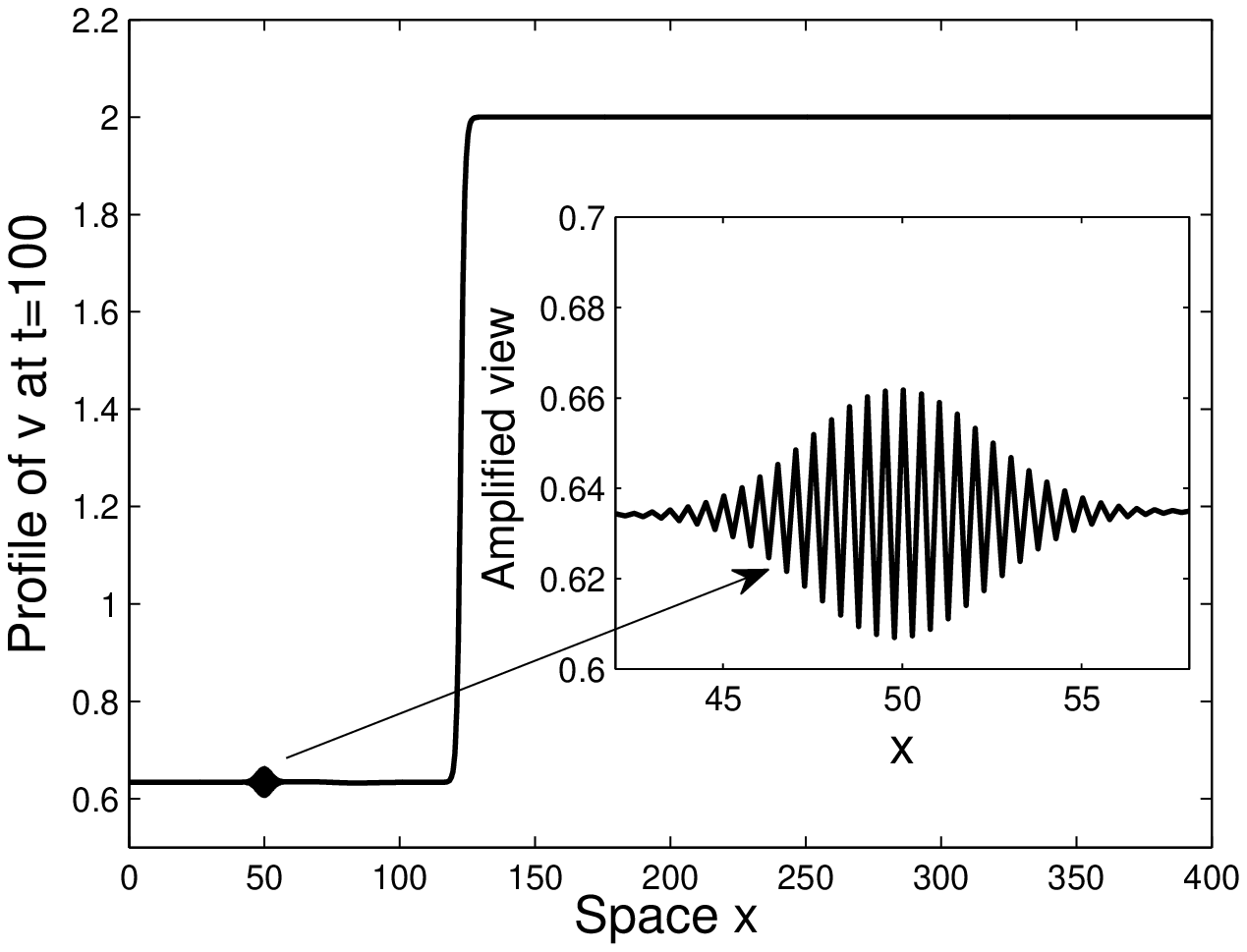}

\includegraphics[width=7cm]{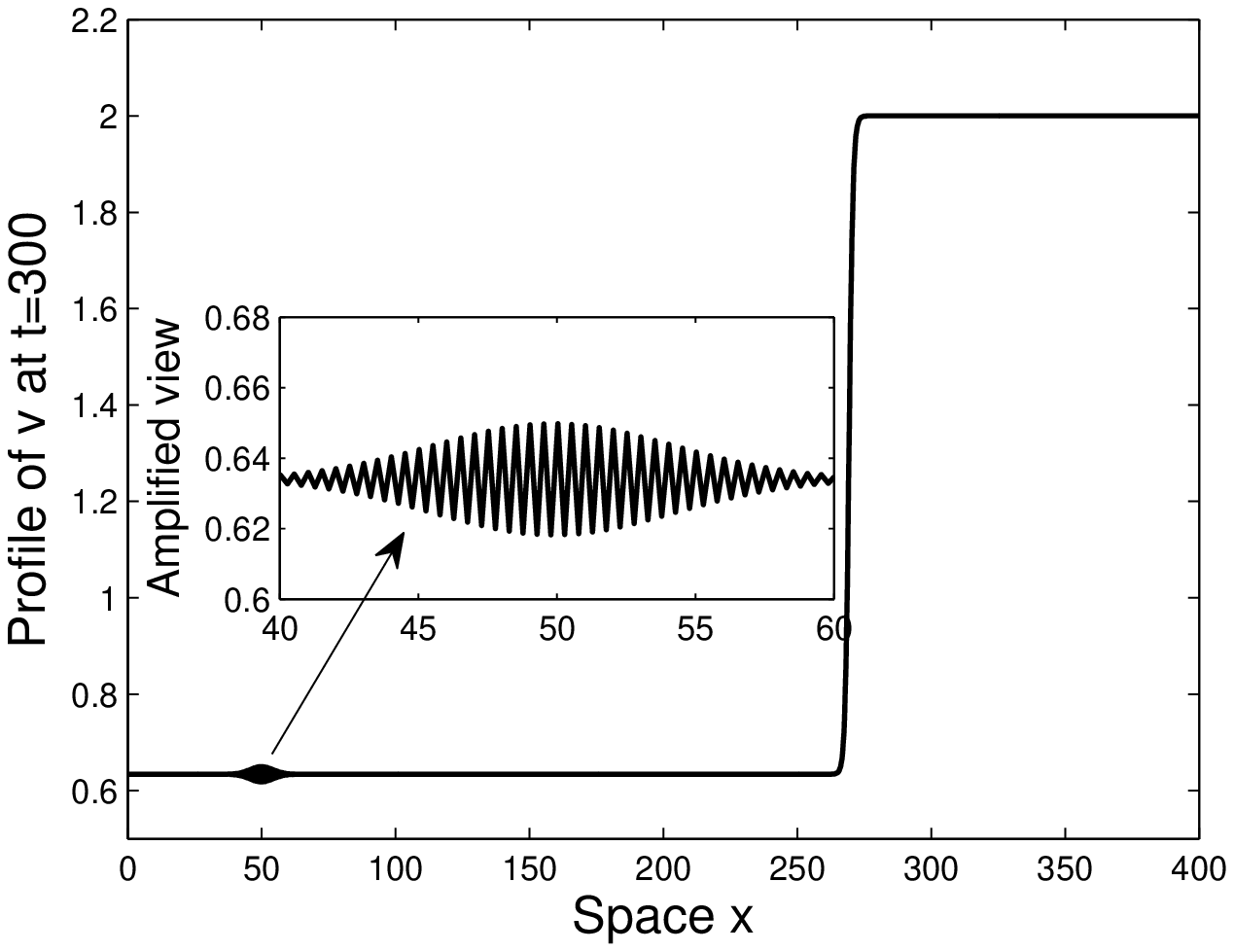}
\includegraphics[width=7cm]{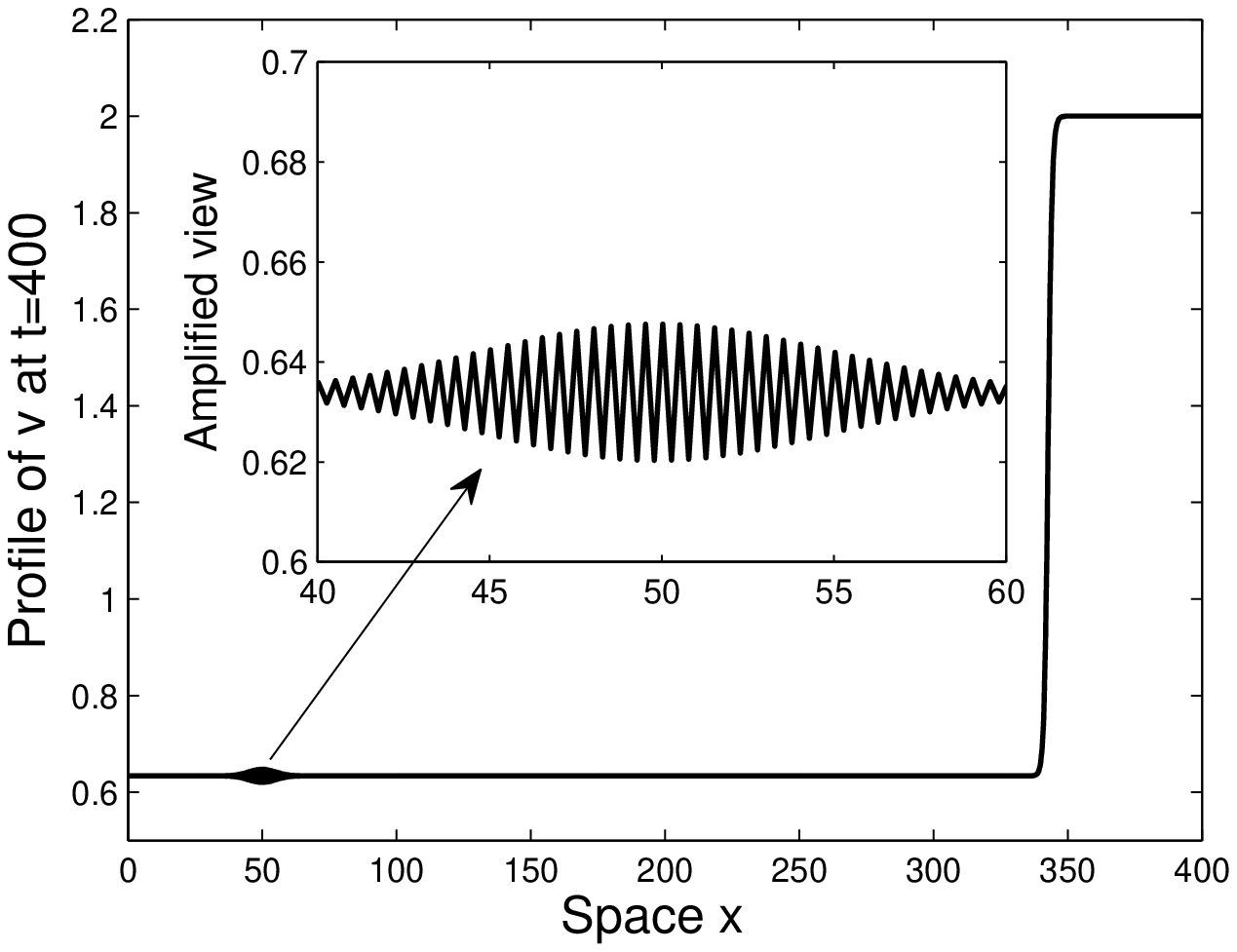}

\caption{The amplified visualization of the fuzzy part for $v$-profile near the discontinuous point $x=50$ plotted in Fig.\ref{fig1}, where we see that traveling wave profile $v$ is not differentiable near the initial discontinuous point $x=50$ and the size of non-differentiability is expanding in space as time evolves.}
\label{fig2}
\end{figure}
Then we prescribe Dirichlet boundary conditions compatible with the asymptotic states $u_{\pm}$ and $v_{\pm}$ and solve the parabolic-hyperbolic system \eqref{hp} with Matlab PDE solver based on the finite-difference method. The numerical solution profiles at progressive time steps are plotted in Fig.\ref{fig1} where we observe that solution component $u$ is smooth for any $t>0$ and converges to a shifted traveling wave profile as $t\to \infty$ but $v$ is non-smooth for any $t>0$ and converge to a shifted traveling wave profile which is non-smooth around the initial discontinuous point $x=50$. This is well consistent with our analytical results of Theorem \ref{u_+=0}. By an amplified view of the non-smoothness near the discontinuous point $x=50$ shown in Fig.\ref{fig2}, one can find that the solution component $v$ is still continuous but not differentiable and furthermore the non-smoothness is expanding in space as time evolves. Our simulations imply that although the parabolic-hyperbolic system \eqref{hp} has a dissipative (parabolic) effect, the classical solutions of (\ref{hp}) seem to be impossible due to the hyperbolic effect and only weak solutions can be established if initial data have only $L^p$-regularity. But the regularity of solution component $u$ can be slightly improved from discontinuity to continuity (see Theorem \ref{u_+=0}). The results in \cite{jin13,Li09} have shown that if the initial value has $H^1$-regularity, then the solution $(u,v)$ may have the same regularity as initial data. However it is unknown if the regularity of solutions can be improved further. Here we exploit the possibility through numerical simulations. To this end, we choose initial value $(u_0,v_0)$ as
\begin{eqnarray}\label{inin}
u_0(x)=\begin{cases}2, & 0\leq x\leq 20\\-\frac{x}{20}+3, & 20<x<40 \\1, & 50< x\leq 300  \end{cases},
\ \ v_0(x)=\begin{cases}\frac{3-\sqrt{3}}{2}, & 0\leq x\leq 20\\\frac{1+\sqrt{3}}{40}x+1-\sqrt{3}, & 20<x<40 \\2, & 20< x\leq 300\end{cases}
\end{eqnarray}
which is continuous and has $H^1$-regularity as plotted in the first panel of Fig.\ref{fig3}. It turns out from numerical results shown in Fig.\ref{fig3} that smooth solutions can be obtained for $(u_0,v_0)$ defined in \eqref{inin}. This finding indicates that the parabolic effect in the system \eqref{hp} can dominate over hyperbolic effect to smoothen  solutions if the initial value has $H^1$-regularity.  However, this is no longer the case once the regularity of initial values is reduced to be discontinuity, as shown in Fig.\ref{fig1}.  {\color{black}Hence these numerical simulations allow us to speculate that the minimal regularity of initial data leading to classical (smooth) solutions of \eqref{hp} is perhaps $H^1$ and the hyperbolic effect of system \eqref{hp} will play an important role when the initial value has regularity lower than $H^1(\R)$}. In this paper we are unable to prove these regularity properties of solutions found in numerical simulations and our speculations may launch interesting questions to pursue in the future.
\begin{figure}[!htbp]
\centering
\includegraphics[width=7.5cm]{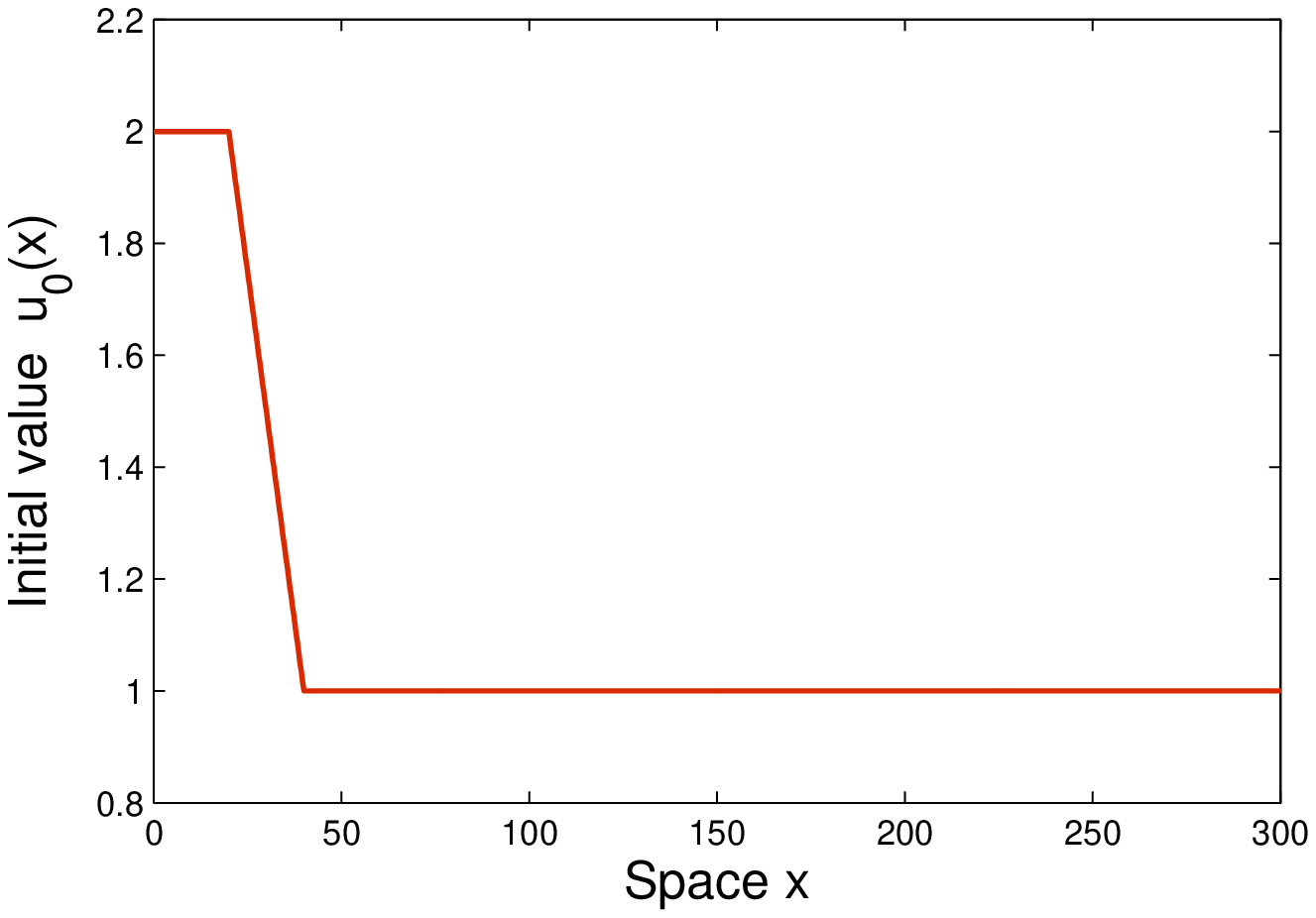}
\includegraphics[width=7.5cm]{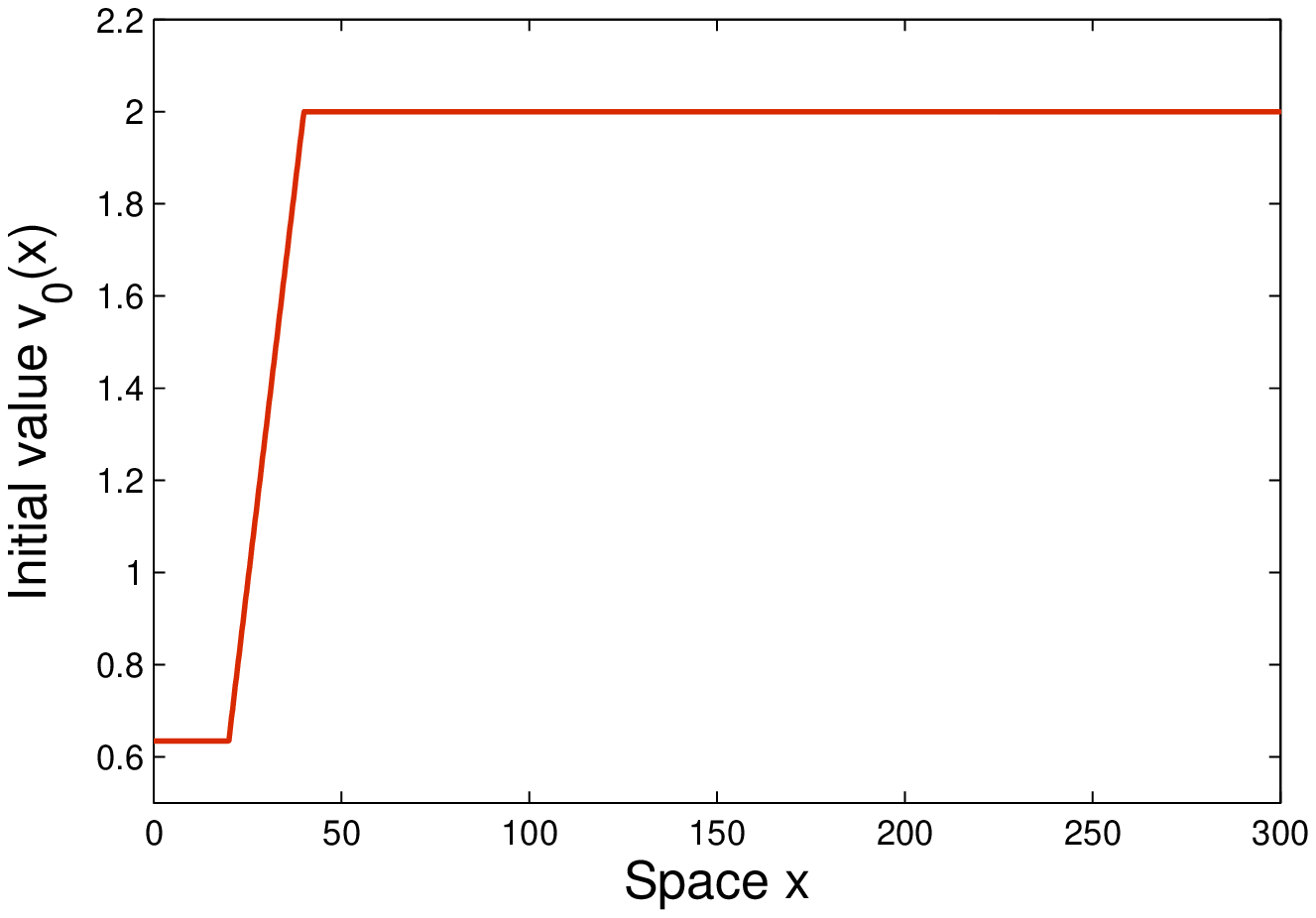}

\includegraphics[width=7.5cm]{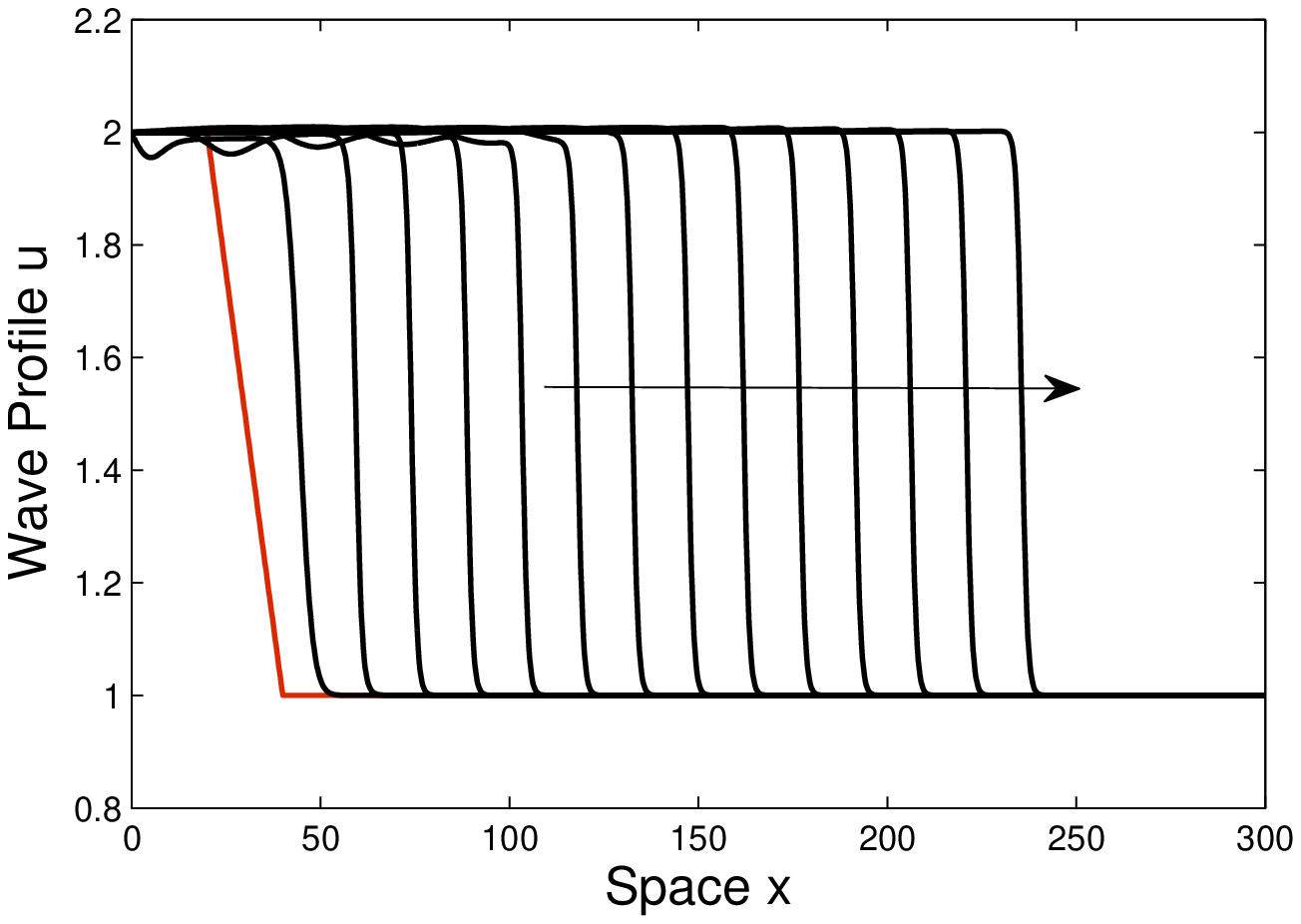}
\includegraphics[width=7.5cm]{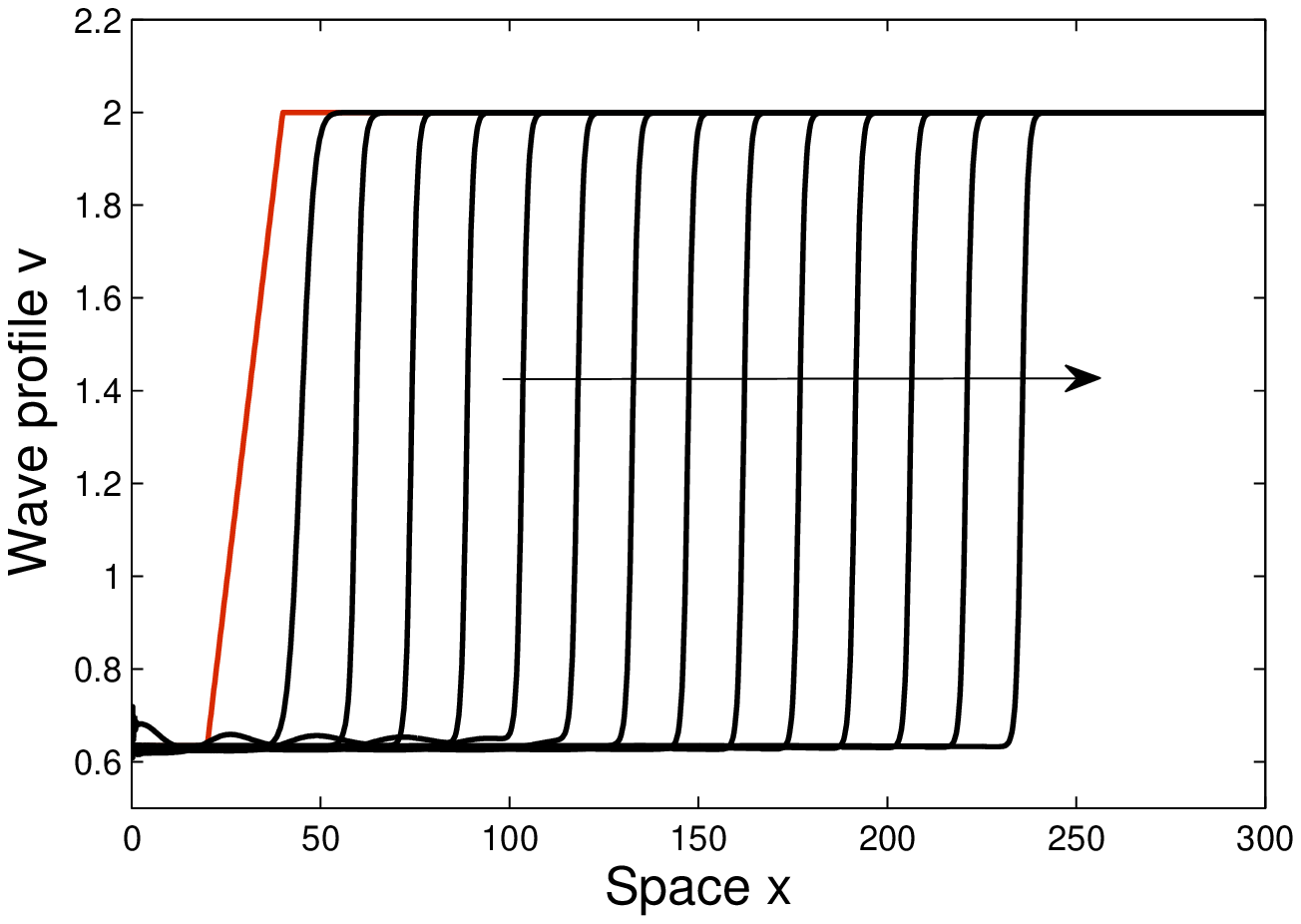}

\caption{Stabilization of smooth traveling waves for \eqref{hp}  with continuous initial data given by \eqref{inin}, where $\chi=1$.}
\label{fig3}
\end{figure}

\bigbreak

\noindent \textbf{Acknowledgement}.  H. Peng was supported by the Fundamental Research Funds
for the Central Universities 2017BQ008. Z. Wang was supported by an internal grant 4-ZZHY from the Hong Kong Polytechnic University.



\begin{thebibliography}{99}
\bibitem{Adams} R.A. Adams and J. J.F. Fournier, {\it Sobolev spaces}, Pure and Applied Mathematics. 140 (2nd ed.). Boston, Academic Press, 2003.
\bibitem{DL} C. Deng and T. Li, Well-posedness of a 3D parabolic-hyperbolic Keller-Segel system in the Sobolev space framework, {\it J. Differential Equations}, 257(2014), 1311-1332.
\bibitem{Fan-zhao} J. Fan and K. Zhao, Blow up criteria for a hyperbolic-parabolic system arising from chemotaxis.
{\it J. Math. Anal. Appl.,} 394(2012), 687-695.
\bibitem{Frid}H. Frid and Y. Li, A boundary value problem for a class of anisotropic degenerate parabolic-hyperbolic equations, {\it Arch. Rational Mech. Anal.}, 226(2017): 975-1008.
\bibitem{Rafael1} R. Granero-Belinch\'{o}n, On the fractional Fisher information with applications to a hyperbolic-parabolic system of chemotaxis, {\it J. Differential Equations}, 262(2017), 3250-3283.
\bibitem{Rafael2} R. Granero-Belinch\'{o}n, Global solutions for a hyperbolic-parabolic system ofchemotaxis, {\it J. Math. Anal. Appl.}, 449(2017), 872-883.
\bibitem{GXZZ} J. Guo, J.X. Xiao, H.J. Zhao and C.J. Zhu, Global solutions to a hyperbolic-parabolic coupled system with large initial data, {\it Acta Math. Sci. Ser. B Engl. Ed.,} 29 (2009), 629-641.
\bibitem{Hao}   C. Hao,  Global well-posedness for a multidimensional chemotaxis model in critical Besov spaces, {\it Z. Angew Math. Phys.}, 63 (2012),  825-834.
\bibitem{HaoJDE} J. Hao, Z. Liu and J. Yong, Regularity analysis for an abstract system of coupled hyperbolic and parabolic equations. {\it J. Differential Equations}, 259(2015), 4763-4798.
\bibitem {Hoff-1987}
 D. Hoff, Global existence for 1D, compressible, isentropic Navier-Stokes equations with large initial data, {\it Trans. Amer. Math. Soc.}, 303(1987), 169-181.

\bibitem {Hoff-19980}
 D. Hoff, Global solutions of the equations of one-dimensional, compressible flow with large data and forces, and with differing end states, {\it Z. Angew. Math. Phys.},
 49(1998), 774-785.

\bibitem {Hoff-1992}
 D. Hoff, Spherically symmetric solutions of the Navier-Stokes equations for compressible, isothermal flow with large, discontinuous initial data,  {\it Indiana Univ. Math. J.},
  41(1992), 1225-1302.

\bibitem {Hoff-jde} D. Hoff, Global solutions of the Navier-Stokes equations for multidimensional compressible
flow with discontinuous initial data, {\it J. Differential
Equations,} 120(1995), 215-254.

\bibitem{Hoff-arma} D. Hoff, Discontinuous solutions of the Navier-Stokes equations for multidimensional flows of
heat-conducting fluids, {\it Arch. Rational Mech. Anal.,} 139(1997),
303-354.

\bibitem {Hoff-1989}
 D. Hoff and T.P. Liu, The inviscid limit for the Navier-Stokes equations of compressible, isentropic flow with shock data, {\it Indiana Univ. Math. J.}, 38(1989), 861-915.

\bibitem{HWJMPA} Q.Q. Hou and Z.A. Wang, Convergence of boundary layers for the Keller-Segel system with singular sensitivity in the half-plane, {\it J. Math. Pures. Appl.}, doi.org/10.1016/j.matpur.2019.01.008, 2019.

\bibitem{jin13}
H.Y. Jin, J.Y. Li and Z.A. Wang, Asymptotic stability of traveling
waves of a chemotaxis model with singular sensitivity, {\it J.
Differential Equations,} 255 (2013), 193-219.

\bibitem{Levine97}
 H.A. Levine and B.D. Sleeman, A system of reaction diffusion equtions arising in the theory of
reinforced random walks, {\it SIAM J. Appl. Math.,} 57(1997),
683-730.
\bibitem{LSN} H.A. Levine, B.D. Sleeman and M. Nilsen-Hamilton,  A mathematical model for the roles of pericytes and macrophages in the initiation of angiogenesis. I. the role of protease inhibitors in preventing angiogenesis,  {\it Math. Biosci.}, 168(2000), 71-115.
\bibitem{Li111}
 D. Li, T. Li and  K. Zhao, On a hyperbolic-parabolic system modeling chemotaxis, {\it Math. Models Methods Appl. Sci.,} 21(2011), 1631-1650.
\bibitem{Li-Zhao-JDE} H. Li and K. Zhao, Initial-boundary value problems for a system of hyperbolic balance laws arising from chemotaxis, {\it J. Differential Equations}, 258(2015), 302-308.
\bibitem{Lipz}
D. Li, R. Pan and K. Zhao, Quantitative decay of a hybrid type chemotaxis model with
large data, {\it Nonlinearity,} 28(2015), 2181-2210.
\bibitem{Lij13}
J.Y. Li, L.N. Wang and K.J. Zhang, Asymptotic stability of a
composite wave of two traveling waves to a hyperbolic-parabolic
system modeling chemotaxis, {\it Math. Methods Appl. Sci.,}
 36(2013), 1862-1877.
\bibitem{Li112}
T. Li, R.H. Pan and K. Zhao, Global dynamics of a chemotaxis model on
bounded domains with large data, {\it SIAM J. Appl. Math.,}
72(2012), 417-443.
\bibitem{Li09}
T. Li and Z.A. Wang, Nonlinear stability of traveling waves to a
hyperbolic-parabolic system modeling chemotaxis, {\it SIAM J. Appl.
Math.,} 70(2009), 1522-1541.
\bibitem{Li10}
T. Li and Z.A. Wang, Nonlinear stability of large amplitude viscous shock waves of a
generalized hyperbolic-parabolic system arising in chemotaxis,
{\it Math. Models Methods Appl. Sci.,} 20(2010), 1967-1998.
\bibitem{Liu09}
T.P. Liu and Y. Zeng, Time-asymptotic behavior of wave propagation around a
viscous shock profile, {\it Commun. Math. Phys.,} 290(2009), 23-82.

\bibitem{Othmer97}
H.G. Othmer and  A. Stevens, Aggregation, blowup, and collapse: the
ABC's of taxis in reinforced random walks, {\it SIAM J. Appl.
Math.,} 57(1997), 1044-1081.
\bibitem{OP} W.S. Oza\'{n}ski and B.C. Pooley, Leray's fundamental work on the Navier-Stokes
equations: a modern review of ``Sur le mouvement d'un liquide visqueux emplissant l'espace'',  arXiv:1708.09787v1, 31 August 2017.


\bibitem{Qin}Y. Qin and L. Huang. {\it Global well-posedness of nonlinear parabolic-hyperbolic coupled systems}. Springer Science and Business Media, 2012.

\bibitem{Rou}T. Roub\'{i}\v{c}ek, {\it Nonlinear Partial Differential Equations with Applications (2nd ed.)},  Basel: Birkh\"{a}user, 2013.
\bibitem{smoller} J. Smoller, {\it Shock Waves and Reaction-Diffusion Equations}, Second edition, Spring-Verlag, Berlin, 1994.

\bibitem{sz93}
A. Szepessy and Z. P. Xin, Nonlinear stability of viscous shock waves, {\it Arch. Rational
Mech. Anal.,} 122(1993), 53-103.

\bibitem{wang08}
Z.A. Wang  and T. Hillen, Shock formation in a chemotaxis model, {\it
Math. Methods Appl. Sci.,} 31 (2008), 45-70.
\bibitem{Wang-xiang-yu}
Z.A. Wang, Z. Xiang and P. Yu,
Asymptotic dynamics on a singular chemotaxis system modeling onset of tumor angiogenesis,
{\it J. Differential Equations,} 260(2016), 2225-2258.
\bibitem{ZhangYH}Y. Zhang, Global analysis of smooth solutions to a hyperbolic-parabolic coupled system, {\it Front. Math. China}, 8:1437-1460, 2013.
\bibitem{zhang07}
M. Zhang and C.J. Zhu, Global existence of solutions to a
hyperbolic-parabolic system, {\it Proc. Amer. Math. Soc.,}  135
(2007), 1017-1027.
\bibitem{zhang131}
Y. Zhang, R.H. Pan and Z. Tan, Zero dissipation limit to a Riemann solution consisting of two shock waves for the 1D compressible isentropic Navier-Stokes equations, {\it Sci. China Math.,} 56 (2013), 2205-2232.
\bibitem{zhang132}
 Y. Zhang, R.H. Pan, Y. Wang and Z. Tan, Zero dissipation limit with two interacting shocks of the 1D non-isentropic Navier-Stokes equations,
 {\it Indiana Univ. Math. J.,} 62(2013), 249-309.
\bibitem{zhangts}
Y. Zhang, Z. Tan and M.B. Sun, Global existence and asymptotic behavior of smooth
solutions to a coupled hyperbolic-parabolic system, {\it Nonlinear Analysis: Real World Applications,}
14(2013), 465-482.
\bibitem{Zheng}S. Zheng. {\it Nonlinear parabolic equations and hyperbolic-parabolic coupled systems}. CRC Press, 1995.
\end{thebibliography}
\end{document}